\documentclass[11pt]{article}
\usepackage{verbatim}
\usepackage{amssymb,amsmath,amsfonts,amsthm,mathrsfs}
\usepackage{graphicx,graphics,psfrag,epsfig,calrsfs,cancel}
\usepackage[colorlinks=true, pdfstartview=FitV, linkcolor=blue, citecolor=red, urlcolor=blue]{hyperref}
\hypersetup{breaklinks=true}

\usepackage{geometry}
\newtheorem{lemma}{Lemma}
\newtheorem{proposition}{Proposition}

\newtheorem{corollary}{Corollary}

\newcommand{\A}{{\mathbb A}}

\newcommand{\R}{{\mathbb R}}
\newcommand{\C}{{\mathbb C}}

\newcommand{\ceta}{{\check\eta}}
\newcommand{\ueta}{\,{\eta}\!\!\!\!\underline{\;\,}\;}
\newcommand{\uceta}{\,{\check\eta}\!\!\!\!\underline{\;\,}\;}
\newcommand{\ubeta}{\,{\beta}\!\!\!\underline{\;\,}\,}

\newcommand{\ua}{\underline{a}}

\newcommand{\dif}{{\rm d}}

\begin{document}

\title{The amplitude equation\\ for weakly nonlinear reversible phase boundaries}

\author{Sylvie {\sc Benzoni-Gavage}$^\dag$ \& Jean-Fran\c{c}ois {\sc Coulombel}$^\ddag$\\
$ $\\
{\small $\dag$ Universit\'e de Lyon, Universit\'e Claude Bernard Lyon 1,}\\
{\small CNRS, UMR5208, Institut Camille Jordan, 43 boulevard du 11 novembre 1918}\\
{\small F-69622 Villeurbanne-Cedex, France}\\
{\small $\ddag$ CNRS, Universit\'e de Nantes, Laboratoire de Math\'ematiques Jean Leray (CNRS UMR6629)}\\
{\small 2 rue de la Houssini\`ere, BP 92208, 44322 Nantes Cedex 3, France}\\
{\small Emails: {\tt benzoni@math.univ-lyon1.fr, jean-francois.coulombel@univ-nantes.fr}}}
\date{\today}
\maketitle

\begin{abstract}
This technical note is a complement to an earlier paper [Benzoni-Gavage \& Rosini, Comput. Math. Appl. 2009], 
which aims at a deeper understanding of a basic model for propagating phase boundaries that was proved to 
admit surface waves [Benzoni-Gavage, Nonlinear Anal. 1998]. The amplitude equation governing the evolution 
of weakly nonlinear surface waves for that model is computed explicitly, and is eventually found to have enough 
symmetry properties for the associated Cauchy problem to be locally well-posed.
\end{abstract}




\section{Introduction}
The reader is assumed to be familiar with the weakly nonlinear theory developed in \cite{BenzoniRosini}, 
which is very much inspired from the seminal work \cite{Hunter}. In particular, we keep the same notation 
as in \cite[Section 2]{BenzoniRosini}, and consider the liquid-vapor phase transition model described in 
\cite[Paragraphs 3.1, 3.2]{BenzoniRosini}. Our goal is to make explicit the amplitude equation, which corresponds 
to \cite[Equation (2.20)]{BenzoniRosini} in an abstract framework. We adopt slightly different conventions compared 
with \cite[Paragraph 3.3]{BenzoniRosini}. Namely, the eigenmodes with nonzero real part for the state ahead of the 
phase transition are\footnote{The main difference here with \cite[Paragraph 3.3]{BenzoniRosini} is our sign convention 
for $a_\ell$, the latter being denoted by $\alpha_\ell$ in \cite[Paragraph 3.3]{BenzoniRosini}.}:
\begin{equation}
\label{valeurspropres1}
\beta_1^- :=\dfrac{a_\ell -i\, u_\ell \, \eta_0}{c_\ell^2 -u_\ell^2} \, ,\quad 
\beta_1^+ :=\dfrac{-a_\ell -i\, u_\ell \, \eta_0}{c_\ell^2 -u_\ell^2} =-\overline{\beta_1^-} \, ,\quad 
a_\ell :=-c_\ell \, \sqrt{(c_\ell^2 -u_\ell^2) \, |\ceta|^2 -\eta_0^2} \, .
\end{equation}
The eigenmodes with nonzero real part for the state behind the phase transition are:
\begin{equation}
\label{valeurspropres2}
\beta_2^- :=\dfrac{-a_r +i\, u_r \, \eta_0}{c_r^2 -u_r^2} \, ,\quad 
\beta_2^+ :=\dfrac{a_r +i\, u_r \, \eta_0}{c_r^2 -u_r^2} =-\overline{\beta_2^-} \, ,\quad 
a_r :=c_r \, \sqrt{(c_r^2 -u_r^2) \, |\ceta|^2 -\eta_0^2} \, .
\end{equation}
The remaining, purely imaginary, eigenmodes are:
\begin{equation*}
\beta_3^+ =\cdots =\beta_{d+1}^+ :=\dfrac{i\, \eta_0}{u_\ell} \, ,\quad 
\beta_3^- =\cdots =\beta_{d+1}^- :=-\dfrac{i\, \eta_0}{u_r} \, .
\end{equation*}
The corresponding right eigenvectors are:
\begin{equation}
\label{vecteurspropres1}
R_1^\pm :=\begin{pmatrix}
r_1^\pm \\
0 \end{pmatrix} \, ,\quad r_1^- :=\begin{pmatrix}
-i\, \eta_0 +u_\ell \, \beta_1^- \\
i \, c_\ell^2 \, \ceta \\
-a_\ell \end{pmatrix} \, ,\quad r_1^+ :=\begin{pmatrix}
i\, \eta_0 -u_\ell \, \beta_1^+ \\
-i \, c_\ell^2 \, \ceta \\
-a_\ell \end{pmatrix} =\overline{r_1^-} \, ,
\end{equation}

\begin{equation}\label{vecteurspropres2}
R_2^\pm :=\begin{pmatrix}
0 \\ 
r_2^\pm \end{pmatrix} \, ,\quad r_2^- :=\begin{pmatrix}
-i\, \eta_0 -u_r \, \beta_2^- \\
i \, c_r^2 \, \ceta \\
-a_r \end{pmatrix} \, ,\quad r_2^+ :=\begin{pmatrix}
i\, \eta_0 +u_r \, \beta_2^+ \\
-i \, c_r^2 \, \ceta \\
-a_r \end{pmatrix} =\overline{r_2^-} \, ,
\end{equation}

\begin{align}
&R_j^+ :=\begin{pmatrix} 
r_j^+ \\
0 \end{pmatrix} \, ,\quad R_j^- :=\begin{pmatrix} 
0 \\
r_j^- \end{pmatrix} \, ,\quad j=3,\dots,d+1 \, ,\notag\\
&r_j^+ :=\begin{pmatrix}
0 \\
\eta_0 \, \check{e}_{j-2} \\
u_\ell \, \ceta \cdot \check{e}_{j-2} \end{pmatrix} \, ,\quad r_j^- :=\begin{pmatrix}
0 \\
\eta_0 \, \check{e}_{j-2} \\
u_r \, \ceta \cdot \check{e}_{j-2} \end{pmatrix} \, ,\label{vecteurspropres3}
\end{align}
with $\check{e}_1 := \ceta$ and the $d-2$ vectors $\check{e}_2,\dots,\check{e}_{d-1} \in \R^{d-1}$ span 
$\ceta^\bot$.

The left eigenvectors are:
\begin{align*}
L_1^\pm :=\begin{pmatrix}
\ell_1^\pm \\
0 \end{pmatrix} \, ,\quad &\ell_1^- :=\dfrac{c_\ell^2 -u_\ell^2}{2\, a_\ell \, (u_\ell \, a_\ell +i\, c_\ell^2 \, \eta_0)} \, 
\begin{pmatrix}
i\, \eta_0 -2\, u_\ell \, \beta_1^+ \\
-i \, \ceta \\
\beta_1^+ \end{pmatrix} \, ,\\
&\ell_1^+ :=\dfrac{c_\ell^2 -u_\ell^2}{2\, a_\ell \, (u_\ell \, a_\ell -i\, c_\ell^2 \, \eta_0)} \, 
\begin{pmatrix}
-i\, \eta_0 +2\, u_\ell \, \beta_1^- \\
i \, \ceta \\
-\beta_1^- \end{pmatrix} =\overline{\ell_1^-} \, ,
\end{align*}

\begin{align*}
L_2^\pm :=\begin{pmatrix}
0 \\
\ell_2^\pm \end{pmatrix} \, ,\quad &\ell_2^- :=\dfrac{c_r^2 -u_r^2}{2\, a_r \, (u_r \, a_r +i\, c_r^2 \, \eta_0)} \, 
\begin{pmatrix}
-i\, \eta_0 -2\, u_r \, \beta_2^+ \\
i \, \ceta \\
\beta_2^+ \end{pmatrix} \, ,\\
&\ell_2^+ :=\dfrac{c_r^2 -u_r^2}{2\, a_r \, (u_r \, a_r -i\, c_r^2 \, \eta_0)} \, \begin{pmatrix}
i\, \eta_0 +2\, u_r \, \beta_2^- \\
-i \, \ceta \\
-\beta_2^- \end{pmatrix} =\overline{\ell_2^-} \, ,
\end{align*}

\begin{align*}
&L_j^+ :=\begin{pmatrix} 
\ell_j^+ \\
0 \end{pmatrix} \, ,\quad L_j^- :=\begin{pmatrix} 
0 \\
\ell_j^- \end{pmatrix} \, ,\quad j=3,\dots,d+1 \, ,\\
&\ell_3^+ :=\dfrac{1}{\eta_0^2 +u_\ell^2 \, |\ceta|^2} \, \begin{pmatrix}
u_\ell \\
-\eta_0/(u_\ell \, |\ceta|^2) \, \ceta \\
-1 \end{pmatrix} \, ,\quad 
\ell_3^- :=\dfrac{1}{\eta_0^2 +u_r^2 \, |\ceta|^2} \, \begin{pmatrix}
-u_r \\
\eta_0/(u_r \, |\ceta|^2) \, \ceta \\
1 \end{pmatrix} \, ,\\
&\ell_j^+ :=-\dfrac{1}{u_\ell \, \eta_0} \, \begin{pmatrix}
0 \\
\check{e}_{j-2}' \\
0 \end{pmatrix} \, ,\quad 
\ell_j^- :=\dfrac{1}{u_r \, \eta_0} \, \begin{pmatrix}
0 \\
\check{e}_{j-2}' \\
0 \end{pmatrix} \, ,\quad j=4,\dots,d+1 \, ,
\end{align*}
where the $d-2$ vectors $\check{e}_2',\dots,\check{e}_{d-1}'$ belong to $\ceta^\bot$ and form the dual basis of 
$\check{e}_2,\dots,\check{e}_{d-1}$. Unlike the choice in \cite[page 1476]{BenzoniRosini}, the left eigenvectors 
here satisfy the normalization property:
\begin{equation*}
\big( L_i^\pm \big)^* \, \breve{\A}^d (\underline{v}) \, R_j^\pm =\delta_{i,j} \, ,\quad 
\big( L_i^\pm \big)^* \, \breve{\A}^d (\underline{v}) \, R_j^\mp =0 \, .
\end{equation*}

After linearizing the jump conditions we are left with the matrices
\begin{equation}
\label{defH}
H(\underline{v}) :=\begin{pmatrix}
0 & 0 & 1 & 0 & 0 & -1 \\
0 & u_\ell \, I_{d-1} & 0 & 0 & -u_r \, I_{d-1} & 0 \\
c_\ell^2 -u_\ell^2 & 0 & 2\, u_\ell & -(c_r^2 -u_r^2) & 0 & -2\, u_r \\
(c_\ell^2 -u_\ell^2) \, u_\ell & 0 & u_\ell^2 +\mu & -(c_r^2 -u_r^2) \, u_r & 0 & -u_r^2 -\mu \end{pmatrix} \, ,
\end{equation}
where
\begin{equation}
\label{defmu}
\mu := \dfrac{1}{2} \, u_\ell^2 +g (\rho_\ell) = \dfrac{1}{2} \, u_r^2 +g (\rho_r) \, ,
\end{equation}
and
\begin{equation}
\label{defj}
J(\underline{v}) \, \eta := \begin{pmatrix}
[\rho] \, \eta_0 \\
[p] \, \ceta \\
0 \\
(\mu \, [\rho] -[p]) \, \eta_0 \end{pmatrix} \, .
\end{equation}
Here we have used the relation
\begin{equation*}
\dfrac{1}{2} \, j \, [u] +[\rho \, f] =\mu \, [\rho] -[p] \, ,
\end{equation*}
where $\mu$ is defined in \eqref{defmu}, in order to simplify the last entry of $J(\underline{v}) \, \eta$.

The Lopatinskii determinant is defined by
\begin{equation}
\label{defdelta}
\Delta (\eta) := \det \begin{pmatrix}
J(\underline{v})\, \eta & H(\underline{v}) \, R_1^- & \dots & H(\underline{v}) \, R_{d+1}^- \end{pmatrix} \, .
\end{equation}
Each column vector in the above determinant is computed by using \eqref{defj}, \eqref{defH} and 
\eqref{vecteurspropres1}, \eqref{vecteurspropres2}, \eqref{vecteurspropres3} for the definition of the right 
eigenvectors. Then some simple manipulations on the rows and columns of the above determinant yield 
(see \cite{Benzoni1998} for similar computations):
\begin{equation*}
\Delta (\eta) =(-u_r \, \eta_0)^{d-2} \, u_r \, \det \begin{pmatrix}
\check{e}_1 & \dots & \check{e}_{d-1} \end{pmatrix} \, 
\det \begin{pmatrix}
[\rho] \, \eta_0 & a_\ell & a_r & |\ceta|^2 \\
0 & u_\ell \, a_\ell +i\, \eta_0 \, c_\ell^2 & u_r \, a_r +i\, \eta_0 \, c_r^2 & 2\, u_r \, |\ceta|^2 \\
-[p] \, \eta_0 & i\, \eta_0 \, u_\ell \, c_\ell^2 & i\, \eta_0 \, u_r \, c_r^2 & u_r^2 \, |\ceta|^2 \\
[p] & -i\, u_\ell \, c_\ell^2 & -i\, u_r \, c_r^2 & \eta_0 \end{pmatrix} \, .
\end{equation*}
For future use, we introduce the quantity
\begin{equation}
\label{defupsilon}
\Upsilon :=(-u_r \, \eta_0)^{d-2} \, u_r \, \det \begin{pmatrix}
\check{e}_1 & \dots & \check{e}_{d-1} \end{pmatrix} \in \R \setminus \{ 0\} \, .
\end{equation}
The expression of the Lopatinskii determinant then reduces to
\begin{equation}
\label{delta}
\Delta (\eta) =-[\rho] \, \, [u] \, \Upsilon \, (\eta_0^2 +u_r^2 \, |\ceta|^2) \, 
\big( u_\ell \, u_r \, a_\ell \, a_r +c_\ell^2 \, c_r^2 \, \eta_0^2 \big) \, .
\end{equation}
We fix a root $\ueta$ of the Lopatinskii determinant (see \cite{Benzoni1998} for the properties 
of such roots, in particular the location of $|\ueta_0|/|\uceta|$ with respect to various velocities associated 
with the phase transition). From now on, an underline refers to evaluation at the root $\ueta$ of the Lopatinskii 
determinant.
\bigskip

We now define a vector $\sigma \in \C^{d+2}$ by computing some of the minors of $\Delta(\ueta)$. More precisely, 
the vectors $H(\underline{v})\, \underline{R}_1^-,\dots,H(\underline{v})\, \underline{R}_{d+1}^-$ are linearly 
independent and we can thus define a vector $\sigma \in \C^{d+2} \setminus \{ 0\}$ satisfying
\begin{equation*}
\forall \, X \in \C^{d+2} \, ,\quad \det \begin{pmatrix}
X & H(\underline{v}) \, \underline{R}_1^- & \dots & H(\underline{v}) \, \underline{R}_{d+1}^- \end{pmatrix} 
=\sigma^* \, X \, .
\end{equation*}
This vector $\sigma$ can be computed explicitly by performing some elementary manipulations on each minor of 
\eqref{defdelta}. We do not give the detailed calculations but rather give the expression of $\sigma$. We find
\begin{equation}
\label{defsigma}
\sigma^* =\underline{\Upsilon} \, \begin{pmatrix}
D_1 & \check{D} \, \uceta^T & D_{d+1} & D_{d+2} \end{pmatrix} \, ,
\end{equation}
where $\underline{\Upsilon}$ denotes the quantity $\Upsilon$ in \eqref{defupsilon} evaluated at the frequency 
$\ueta$, and\footnote{Since $\ua_\ell$ is negative and $\ua_r$ is positive, $D_{d+1}$ is nonzero and we can 
thus check that $\sigma$ is a nonzero vector.}
\begin{align}
\label{defD}
\begin{split}
D_1 +\mu \, D_{d+2} &:= -(\ueta_0^2 +u_r^2 \, |\uceta|^2) \, \big( [u] \, c_\ell^2 \, c_r^2 \, \ueta_0 
-i \, u_\ell \, u_r \, (c_r^2 \, \ua_\ell -c_\ell^2 \, \ua_r) \big) \, ,\\
\check{D} &:= -[u] \, u_r \, \big( \ua_\ell \, (u_r \, \ua_r -i\, c_r^2 \, \ueta_0) 
+\ua_r \, (u_\ell \, \ua_\ell -i\, c_\ell^2 \, \ueta_0) \big) \, ,\\
D_{d+1} &:= -i\, (\ueta_0^2 +u_r^2 \, |\uceta|^2) \, (u_r \, c_r^2 \, \ua_\ell -u_\ell \, c_\ell^2 \, \ua_r) \, ,\\
D_{d+2} &:= [u] \, \ueta_0 \, (\ua_\ell \, \ua_r +c_\ell^2 \, c_r^2 \, |\uceta|^2) \\
&\quad \quad +i\, \ueta_0^2 \, (c_r^2 \, \ua_\ell -c_\ell^2 \, \ua_r) 
-i\, |\uceta|^2 \, \big( u_r \, (u_\ell-2\, u_r) \, c_r^2 \, \ua_\ell +u_\ell \, u_r \, c_\ell^2 \, \ua_r \big) \, .
\end{split}
\end{align}

It remains to compute the coefficients $\gamma_1,\gamma_2$ satisfying:
\begin{equation*}
J(\underline{v})\, \ueta +\gamma_1 \,H(\underline{v}) \, \underline{R}_1^- 
+\gamma_2 \,H(\underline{v}) \, \underline{R}_2^- =0 \, .
\end{equation*}
Observe that our convention differs from that in \cite[page 1477]{BenzoniRosini}. We get:
\begin{equation}
\label{defgamma}
\gamma_1 :=\dfrac{[\rho] \, u_r \, \ueta_0}{u_r \, \ua_\ell -i\, c_\ell^2 \, \ueta_0} 
=\dfrac{-i \, [\rho] \, c_r^2 \, \ueta_0^2}{\ua_\ell \, (u_\ell \, \ua_r -i\, c_r^2 \, \ueta_0)} \, ,\quad 
\gamma_2 :=\dfrac{-[\rho] \, u_\ell \, \ueta_0}{u_\ell \, \ua_r -i\, c_r^2 \, \ueta_0} 
=\dfrac{i \, [\rho] \, c_\ell^2 \, \ueta_0^2}{\ua_r \, (u_r \, \ua_\ell -i\, c_\ell^2 \, \ueta_0)} \, ,
\end{equation}
where the equalities follow from the relation $u_\ell \, u_r \, \ua_\ell \, \ua_r +c_\ell^2 \, c_r^2 \, \ueta_0^2=0$ that is 
satisfied by the root $\ueta$ of the Lopatinskii determinant. For notational convenience, we also set
\begin{equation*}
\gamma_3 =\cdots =\gamma_{d+1} := 0 \, .
\end{equation*}
\bigskip

Following \cite[Proposition 2.2]{BenzoniRosini}, the evolution of a weakly nonlinear phase transition is governed by 
a scalar amplitude $w$ obeying a nonlocal Burgers equation:
\begin{equation}
\label{BurgersEuler}
a_0(k) \, \partial_\tau \widehat{w} (\tau,k) 
+\int_\R a_1(k-k',k') \, \widehat{w} (\tau,k-k') \, \widehat{w} (\tau,k') \, {\rm d}k' =0 \, ,
\end{equation}
where $a_0$ and $a_1$ are given by Equations (2.24) and (2.25) in \cite[page 1471]{BenzoniRosini}. With the 
present notation, this yields
\begin{equation*}
a_0(k) =\begin{cases}
\alpha_0/(i\, k) &\text{\rm if $k>0$,} \\
\overline{\alpha_0}/(i\, k) &\text{\rm if $k<0$,}
\end{cases}
\end{equation*}
and $\alpha_0$ is a complex number whose definition is recalled in Equation \eqref{defalpha0} below. The 
expression of the kernel $a_1$ is recalled and made explicit in Section \ref{sect2} below.

\section{Computation of the coefficient $\alpha_0$}

\begin{proposition}
\label{prop1}
The coefficient $\alpha_0$ in the expression of $a_0$ is given by
\begin{equation}
\label{alpha0}
\alpha_0 =-\dfrac{[\rho] \, \, [u] \, \underline{\Upsilon}}{\ueta_0} \, (\ueta_0^2 +u_r^2 \, |\uceta|^2) \, 
\left\{ u_\ell^2 \, u_r^2 \, \left( \dfrac{\ua_\ell^2}{c_\ell^2} +\dfrac{\ua_r^2}{c_r^2} \right) 
+2\, c_\ell^2 \, c_r^2 \, \ueta_0^2 \right\} \, ,
\end{equation}
and it coincides with the derivative of the Lopatinskii determinant $\Delta$ with respect to $\eta_0$ at its 
root $\ueta$. In particular, $\alpha_0$ is a nonzero real number.
\end{proposition}

\begin{proof}
We recall that the expression of $\alpha_0$ is
\begin{equation}
\label{defalpha0}
\alpha_0 =\sigma^* \, [\tilde{f}_0(\underline{v})] +i \, \sigma^* \, H(\underline{v}) \, \underline{R}_p^+ \, 
\dfrac{(\underline{L}_p^+)^* \, \gamma_q \, \underline{R}_q^-}{\ubeta_p^+ -\ubeta_q^-} \, ,
\end{equation}
where we use Einstein's summation convention over repeated indices. We first observe that the Hermitian product 
$(\underline{L}_p^+)^* \, \underline{R}_q^-$, $q=1,2$, vanishes as soon as $p$ is larger than $4$. In the same way, 
the products $(\underline{L}_1^+)^* \, \underline{R}_2^-$, $(\underline{L}_3^+)^* \, \underline{R}_2^-$ and 
$(\underline{L}_2^+)^* \, \underline{R}_1^-$ vanish so the expression of $\alpha_0$ reduces to the sum of four terms:
\begin{equation}
\label{decompalpha0}
\alpha_0 =\sigma^* \, [\tilde{f}_0(\underline{v})] 
+i \, \sigma^* \, H(\underline{v}) \, \underline{R}_2^+ \, 
\dfrac{(\underline{L}_2^+)^* \, \gamma_2 \, \underline{R}_2^-}{\ubeta_2^+ -\ubeta_2^-} 
+i \, \sigma^* \, H(\underline{v}) \, \underline{R}_1^+ \, 
\dfrac{(\underline{L}_1^+)^* \, \gamma_1 \, \underline{R}_1^-}{\ubeta_1^+ -\ubeta_1^-} 
+i \, \sigma^* \, H(\underline{v}) \, \underline{R}_3^+ \, 
\dfrac{(\underline{L}_3^+)^* \, \gamma_1 \, \underline{R}_1^-}{\ubeta_3^+ -\ubeta_1^-} \, .
\end{equation}
We now compute each of these four quantities separately. Using
\begin{equation*}
[\tilde{f}_0(\underline{v})] =\dfrac{1}{\ueta_0} \, J(\underline{v}) \, \ueta -\dfrac{1}{\ueta_0} \, \begin{pmatrix}
0 \\
[p] \, \uceta \\
0 \\
0 \end{pmatrix} \, ,
\end{equation*}
and the orthogonality relation $\sigma^* \, J(\underline{v}) \, \ueta=0$, we get
\begin{align}
\sigma^* \, [\tilde{f}_0(\underline{v})] 
&=-[\rho] \, \underline{\Upsilon} \, \dfrac{u_\ell \, u_r \, |\uceta|^2}{\ueta_0} \, \check{D} \notag \\
&= [\rho] \, [u] \, \underline{\Upsilon} \, \left\{  \dfrac{u_\ell \, u_r^2 \, |\uceta|^2}{\ueta_0} \, 
\big( \ua_\ell \, (u_r \, \ua_r -i\, c_r^2 \, \ueta_0) +\ua_r \, (u_\ell \, \ua_\ell -i\, c_\ell^2 \, \ueta_0) \big) 
\right\} \, .\label{prop11}
\end{align}

We now turn to the second term on the right in \eqref{decompalpha0}. There holds
\begin{equation*}
\dfrac{(\underline{L}_2^+)^* \, \gamma_2 \, \underline{R}_2^-}{\ubeta_2^+ -\ubeta_2^-} 
=-\dfrac{c_r^2 \, |\uceta|^2 \, (u_r \, \ua_r -i\, c_r^2 \, \ueta_0)}{2 \, \ua_r^2 \, (\ueta_0^2 +u_r^2 \, |\uceta|^2)} 
\, \gamma_2 \, ,
\end{equation*}
and we also compute
\begin{align}
i \, \sigma^* \, H(\underline{v}) \, \underline{R}_2^+ =i \, \sigma^* \, \overline{H(\underline{v}) \, \underline{R}_2^-} &= 
2 \, \underline{\Upsilon} \, c_r^2 \, \Big\{ \ueta_0 \, (D_{d+1} +u_r \, D_{d+2}) -u_r \, |\uceta|^2 \check{D} \Big\} \notag \\
&= 2 \, [u] \, \underline{\Upsilon} \, c_r^2 \, \ua_r \, (\ueta_0^2 +u_r^2 \, |\uceta|^2) \, 
(u_r \, \ua_\ell -i\, c_\ell^2 \, \ueta_0) \, ,\label{prod1}
\end{align}
where the latter relation is obtained by using the expressions \eqref{defD}. Recalling the expression \eqref{defgamma} 
of $\gamma_2$, we obtain
\begin{equation}
\label{prop12}
i \, \sigma^* \, H(\underline{v}) \, \underline{R}_2^+ \, 
\dfrac{(\underline{L}_2^+)^* \, \gamma_2 \, \underline{R}_2^-}{\ubeta_2^+ -\ubeta_2^-} 
=[\rho] \, \, [u] \, \underline{\Upsilon} \, \left\{ i \, \dfrac{u_\ell \, u_r \, \ua_\ell}{\ua_r} \, c_r^2 \, |\uceta|^2 \, 
(u_r \, \ua_r -i\, c_r^2 \, \ueta_0) \right\} \, .
\end{equation}

We now examine the third term on the right in \eqref{decompalpha0}. There holds:
\begin{equation*}
\dfrac{(\underline{L}_1^+)^* \, \gamma_1 \, \underline{R}_1^-}{\ubeta_1^+ -\ubeta_1^-} 
=-\dfrac{c_\ell^2 \, |\uceta|^2 \, (u_\ell \, \ua_\ell -i\, c_\ell^2 \, \ueta_0)}{2 \, \ua_\ell^2 \, (\ueta_0^2 +u_\ell^2 \, |\uceta|^2)} 
\, \gamma_1 \, ,
\end{equation*}
and we also compute
\begin{align}
i \, \sigma^* \, H(\underline{v}) \, \underline{R}_1^+ =i \, \sigma^* \, \overline{H(\underline{v}) \, \underline{R}_1^-} &= 2\, 
\underline{\Upsilon} \, c_\ell^2 \, \Big\{ u_\ell \, |\uceta|^2 \check{D} -\ueta_0 \, (D_{d+1} +u_\ell \, D_{d+2}) \Big\} \notag \\
&= -2 \, [u] \, \underline{\Upsilon} \, c_\ell^2 \, \ua_\ell \, (\ueta_0^2 +u_r^2 \, |\uceta|^2) \, 
(u_\ell \, \ua_r -i\, c_r^2 \, \ueta_0) \, .\label{prod2}
\end{align}
Using the expression \eqref{defgamma} of $\gamma_1$, we obtain
\begin{equation}
\label{prop13}
i \, \sigma^* \, H(\underline{v}) \, \underline{R}_1^+ \, 
\dfrac{(\underline{L}_1^+)^* \, \gamma_1 \, \underline{R}_1^-}{\ubeta_1^+ -\ubeta_1^-} 
=[\rho] \, \, [u] \, \underline{\Upsilon} \, \left\{ i \, \dfrac{u_\ell \, u_r \, \ua_r}{\ua_\ell} \, c_\ell^2 \, |\uceta|^2 \, 
\dfrac{\ueta_0^2 +u_r^2 \, |\uceta|^2}{\ueta_0^2 +u_\ell^2 \, |\uceta|^2} \, 
(u_\ell \, \ua_\ell -i\, c_\ell^2 \, \ueta_0) \right\} \, .
\end{equation}

It remains to compute the last term on the right in \eqref{decompalpha0} and to add the four expressions. 
First we compute
\begin{equation*}
\dfrac{(\underline{L}_3^+)^* \, \gamma_1 \, \underline{R}_1^-}{\ubeta_3^+ -\ubeta_1^-} 
=-\dfrac{c_\ell^2}{\ueta_0^2 +u_\ell^2 \, |\uceta|^2} \, \gamma_1 \, ,
\end{equation*}
and we also compute
\begin{equation}
\label{prop14}
\sigma^* \, H(\underline{v}) \, \underline{R}_3^+ =-[u] \, \underline{\Upsilon} \, u_\ell \, |\uceta|^2 \, 
(2\, D_{d+1} +(u_\ell +u_r) \, D_{d+2}) \, ,
\end{equation}
where we have used the relation (which amounts to $\sigma^* \, H(\underline{v}) \, \underline{R}_3^- =0$):
\begin{equation*}
D_1 +\ueta_0 \, \check{D} +2\, u_r \, D_{d+1} +(\mu +u_r^2) \, D_{d+2} =0 \, .
\end{equation*}
The expression \eqref{prop14} can be factorized by using the definitions \eqref{defD} of $D_{d+1}$, $D_{d+2}$, 
and we obtain
\begin{equation}
\label{prod3}
\sigma^* \, H(\underline{v}) \, \underline{R}_3^+ =-\dfrac{[u]^2 \, \underline{\Upsilon} \, u_\ell \, |\uceta|^2}{\ueta_0} \, 
(\ueta_0^2 -u_\ell \, u_r \, |\uceta|^2) \, \big( \ua_r \, (u_r \, \ua_\ell -i\, c_\ell^2 \, \ueta_0) 
+\ua_\ell \, (u_\ell \, \ua_r -i\, c_r^2 \, \ueta_0) \big) \, .
\end{equation}
Using \eqref{defgamma}, we derive the expression
\begin{equation}
\label{prop15}
i \, \sigma^* \, H(\underline{v}) \, \underline{R}_3^+ \, 
\dfrac{(\underline{L}_3^+)^* \, \gamma_1 \, \underline{R}_1^-}{\ubeta_3^+ -\ubeta_1^-} 
=[\rho] \, \, [u] \, \underline{\Upsilon} \, \left\{ i \, [u] \, u_\ell \, c_\ell^2 \, |\uceta|^2 \, 
\dfrac{\ueta_0^2 -u_\ell \, u_r \, |\uceta|^2}{\ueta_0^2 +u_\ell^2 \, |\uceta|^2} \, 
(u_r \, \ua_r -i\, c_r^2 \, \ueta_0) \right\} \, .
\end{equation}

According to the decomposition \eqref{decompalpha0}, the coefficient $\alpha_0$ is the sum of the four 
quantities in \eqref{prop11}, \eqref{prop12}, \eqref{prop13} and \eqref{prop15}. Factorizing $[\rho] \, \, [u] 
\, \underline{\Upsilon}$ in each term, we first observe that the imaginary part of the sum equals zero. We 
can thus simplify $\alpha_0$ by retaining only the real part of each term. This leads to the expression
\begin{multline*}
\dfrac{\ueta_0 \, \alpha_0}{[\rho] \, \, [u] \, \underline{\Upsilon}} 
=-c_\ell^2 \, c_r^2 \, \ueta_0^2 \, (u_r^2 +u_\ell \, u_r) \, |\uceta|^2 
+u_\ell \, u_r \, \ua_\ell \, \dfrac{c_r^2 \, \ueta_0^2}{\ua_r} \, c_r^2 \, |\uceta|^2 \\
+u_\ell \, u_r \, \ua_r \, \dfrac{c_\ell^2 \, \ueta_0^2}{\ua_\ell} \, c_\ell^2 \, |\uceta|^2 \, 
\dfrac{\ueta_0^2 +u_r^2 \, |\uceta|^2}{\ueta_0^2 +u_\ell^2 \, |\uceta|^2} 
+(u_\ell \, u_r -u_\ell^2) \, |\uceta|^2 \, c_\ell^2 \, c_r^2 \, \ueta_0^2 \, 
\dfrac{\ueta_0^2 -u_\ell \, u_r \, |\uceta|^2}{\ueta_0^2 +u_\ell^2 \, |\uceta|^2} \, .
\end{multline*}
At this stage, some elementary manipulations lead to the expression \eqref{alpha0} of $\alpha_0$.

The link between $\alpha_0$ and the partial derivative $\partial_{\eta_0} \Delta (\ueta)$ comes from the relation
\begin{equation*}
\ueta_0 \, \dfrac{\partial}{\partial \eta_0} (u_\ell \, u_r \, a_\ell \, a_r +c_\ell^2 \, c_r^2 \, \eta_0^2) \Big|_{\ueta} 
=u_\ell^2 \, u_r^2 \, \left( \dfrac{\ua_\ell^2}{c_\ell^2} +\dfrac{\ua_r^2}{c_r^2} \right) 
+2\, c_\ell^2 \, c_r^2 \, \ueta_0^2 \, ,
\end{equation*}
which is obtained by differentiating \eqref{delta} and the expressions \eqref{valeurspropres1}, 
\eqref{valeurspropres2} with respect to $\eta_0$ and then evaluating at $\ueta$.
\end{proof}

Since the coefficient $\alpha_0$ is real, we obtain $a_0(k)=\alpha_0 /(i\, k)$ for all $k \neq 0$. In particular, 
the amplitude equation \eqref{BurgersEuler} reduces to
\begin{equation*}
\partial_\tau \widehat{w} (\tau,k) +\dfrac{i \, k}{\alpha_0} \, \int_\R a_1(k-k',k') \, \widehat{w} (\tau,k-k') \, 
\widehat{w} (\tau,k') \, \dif k' =0 \, .
\end{equation*}
As far as smooth solutions are concerned, the Cauchy problem associated with this kind of nonlocal Burgers 
equation is known to be locally well-posed under rather simple algebraic conditions (see \cite{Benzoni2009}). 
These conditions are invariant under multiplication by a nonzero real constant, so that it is sufficient to investigate 
whether they are satisfied by the slightly simpler kernel $4\, \pi \, a_1$. This is the purpose of the next section.

\section{Computation of the quadratic kernel}
\label{sect2}

We define the kernel $q(k,k') := 4\, \pi \, a_1(k,k')$. Following \cite{BenzoniRosini}, we can decompose $q$ as 
follows:
\begin{equation*}
q(k,k') =\sum_{j=1}^5 q_j(k,k') \, ,
\end{equation*}
where the kernels $q_1,\dots,q_5$ are given by\footnote{We keep the notation of \cite{BenzoniRosini} for the 
functions $\widehat{r}_\pm,\widehat{r}$ and so on.}
\begin{align}
q_1(k,k') &:= \sigma (k+k')^* \, \sum_{j=0}^{d-1} \ueta_j \, \Big\{ 
\dif\tilde{f}^j (v_r) \cdot (\widehat{r}_+ (k,0) +\widehat{r}_+ (k',0)) 
-\dif\tilde{f}^j (v_\ell) \cdot (\widehat{r}_- (k,0) +\widehat{r}_- (k',0)) \Big\} \, ,\label{defq1} \\
q_2(k,k') &:=  -\sigma (k+k')^* \, \Big\{ \dif^2\tilde{f}^d (v_r) \cdot (\widehat{r}_+ (k,0),\widehat{r}_+ (k',0)) 
-\dif^2\tilde{f}^d (v_\ell) \cdot (\widehat{r}_- (k,0),\widehat{r}_- (k',0)) \Big\} \, ,\label{defq2} \\
q_3(k,k') &:= i\, (k+k') \, \int_0^{+\infty} L(k+k',z) \, \dif \A (v,\ueta) \cdot \widehat{r} (k,z) \cdot 
\widehat{r} (k',z) \, \dif z \, ,\label{defq3} \\
q_4(k,k') &:= \int_0^{+\infty} L(k+k',z) \, \dfrac{\partial}{\partial z} \, \left( \dif \breve{\A}^d (v) \cdot \widehat{r} (k,z) 
\cdot \widehat{r} (k',z) \right) \, \dif z \, ,\label{defq4} \\
q_5(k,k') &:= -\int_0^{+\infty} L(k+k',z) \, \breve{\A} (v,\ueta) \left( \dfrac{\partial \widehat{r}}{\partial z} (k,z) 
+\dfrac{\partial \widehat{r}}{\partial z} (k',z) \right) \, \dif z \, .\label{defq5}
\end{align}
We first examine the kernel $q_1$ and derive its expression for all values of $(k,k')$.

\begin{lemma}
\label{lem1}
Let us define the quantity
\begin{equation}
\label{defQ}
Q:=2\, [\rho] \, [u] \, \underline{\Upsilon} \, (\ueta_0^2 +u_r^2 \, |\uceta|^2) \, (\ubeta_1^- +\ubeta_2^-) \, 
i\, u_\ell \, u_r \, \ua_\ell \, \ua_r \, \dfrac{u_\ell \, \ua_r +i\, c_r^2 \, \ueta_0}{u_\ell \, \ua_r -i\, c_r^2 \, \ueta_0} \, .
\end{equation}
Then the kernel $q_1$ in \eqref{defq1} satisfies
\begin{equation*}
q_1(k,k') =\begin{cases}
0 &\text{\rm if $k>0$ and $k'>0$,}\\
\overline{Q} &\text{\rm if $k>0$, $k'<0$ and $k+k'>0$.}
\end{cases}
\end{equation*}
\end{lemma}

\begin{proof}
The function $g^0$ is known to be an entropy for the isothermal Euler equations with corresponding flux 
$(g^1,\dots,g^d)$. We thus have the relations
\begin{equation*}
\forall \, j=0,\dots,d \, ,\quad \dif g^0(v) \, A^j(v) =\dif g^j(v) \, ,
\end{equation*}
where we use the convention $A^0(v)=I$ for all $v$. Using this relation between the Jacobian matrices, we get
\begin{equation*}
\sum_{j=0}^{d-1} \ueta_j \, \dif\tilde{f}^j (v_r) \cdot \underline{r}_2^- =\begin{pmatrix}
\sum_{j=0}^{d-1} \ueta_j \, A^j (v_r) \, \underline{r}_2^- \\
\dif g^0(v_r) \, \sum_{j=0}^{d-1} \ueta_j \, A^j (v_r) \, \underline{r}_2^- \end{pmatrix} 
=i\, \ubeta_2^- \, \begin{pmatrix}
A^d (v_r) \, \underline{r}_2^- \\
\dif g^d(v_r) \cdot \underline{r}_2^- \end{pmatrix} =-i \, \ubeta_2^- \, H(\underline{v}) \, \underline{R}_2^- \, .
\end{equation*}
Similarly we have
\begin{equation*}
\sum_{j=0}^{d-1} \ueta_j \, \dif\tilde{f}^j (v_\ell) \cdot \underline{r}_1^- 
=-i \, \ubeta_1^- \, H(\underline{v}) \, \underline{R}_1^- \, .
\end{equation*}
For $k>0$ and $k'>0$, we thus get
\begin{align*}
q_1(k,k') &=2\, \sigma^* \, \sum_{j=0}^{d-1} \gamma_2 \, \ueta_j \, \dif\tilde{f}^j (v_r) \cdot \underline{r}_2^- 
-2\, \sigma^* \, \sum_{j=0}^{d-1} \gamma_1 \, \ueta_j \, \dif\tilde{f}^j (v_\ell) \cdot \underline{r}_1^- \\
&=-2\, i \, \gamma_2 \, \ubeta_2^- \, \sigma^* \, H(\underline{v}) \, \underline{R}_2^- 
+2\, i \, \gamma_1 \, \ubeta_1^- \, \sigma^* \, H(\underline{v}) \, \underline{R}_1^- =0 \, ,
\end{align*}
because $\sigma$ is orthogonal to both $H(\underline{v}) \, \underline{R}_1^-$ and 
$H(\underline{v}) \, \underline{R}_2^-$.
\bigskip

Let us now consider the case $k>0$, $k'<0$ and $k+k'>0$. Using the same relations as above for the differentials 
${\rm d}g^j$, we obtain
\begin{align*}
q_1(k,k') &=\sigma^* \, \sum_{j=0}^{d-1} \ueta_j \, \dif\tilde{f}^j (v_r) \cdot 
(\gamma_2 \, \underline{r}_2^- +\overline{\gamma_2} \, \underline{r}_2^+) 
-\sigma^* \, \sum_{j=0}^{d-1} \ueta_j \, \dif\tilde{f}^j (v_\ell) \cdot 
(\gamma_1 \, \underline{r}_1^- +\overline{\gamma_1} \, \underline{r}_1^+) \\
&=-i \, \overline{\gamma_2} \, \ubeta_2^+ \, \sigma^* \, H(\underline{v}) \, \underline{R}_2^+ 
+i \, \overline{\gamma_1} \, \ubeta_1^+ \, \sigma^* \, H(\underline{v}) \, \underline{R}_1^+ \, .
\end{align*}
The Hermitian products $\sigma^* \, H(\underline{v}) \, \underline{R}_2^+$ and $\sigma^* \, H(\underline{v}) \, 
\underline{R}_1^+$ have already been computed in the proof of Proposition \ref{prop1}, see \eqref{prod1} and 
\eqref{prod2}. We then obtain
\begin{align*}
q_1(k,k') =& -2\, [u] \, \underline{\Upsilon} \, (\ueta_0^2 +u_r^2 \, |\uceta|^2) \, c_r^2 \, \ua_r \, 
(u_r \, \ua_\ell -i\, c_\ell^2 \, \ueta_0) \, \overline{\gamma_2} \, \ubeta_2^+ \\
&-2\, [u] \, \underline{\Upsilon} \, (\ueta_0^2 +u_r^2 \, |\uceta|^2) \, c_\ell^2 \, \ua_\ell \, 
(u_\ell \, \ua_r -i\, c_r^2 \, \ueta_0) \, \overline{\gamma_1} \, \ubeta_1^+ \, .
\end{align*}
We use the definition \eqref{defgamma} to obtain
\begin{equation*}
c_r^2 \, \ua_r \, (u_r \, \ua_\ell -i\, c_\ell^2 \, \ueta_0) \, \overline{\gamma_2} 
=c_\ell^2 \, \ua_\ell \, (u_\ell \, \ua_r -i\, c_r^2 \, \ueta_0) \, \overline{\gamma_1} =i\, [\rho] \, 
c_\ell^2 \, c_r^2 \, \ueta_0^2 \, \dfrac{u_\ell \, \ua_r -i\, c_r^2 \, \ueta_0}{u_\ell \, \ua_r +i\, c_r^2 \, \ueta_0} \, ,
\end{equation*}
and the claim follows using the relation $c_\ell^2 \, c_r^2 \, \ueta_0^2 =-u_\ell \, u_r \, \ua_\ell \, \ua_r$.
\end{proof}

Deriving the expression of the kernels $q_3,q_4,q_5$ requires the expression of the row vector 
$L(k+k',z)$, which we derive right now.

\begin{lemma}
\label{lem2}
For $k>0$, there holds
\begin{equation*}
L(k,z) =\begin{pmatrix}
\dfrac{\omega_1}{\gamma_1} \, \exp (-k\, \ubeta_1^+ \, z) \, \widetilde{\ell}_1 
+\dfrac{\omega_3}{\gamma_1} \, \exp (-k\, \ubeta_3^+ \, z) \, \widetilde{\ell}_3 
& \dfrac{\omega_2}{\gamma_2} \, \exp (-k\, \ubeta_2^+ \, z) \, \widetilde{\ell}_2 \end{pmatrix} \, ,
\end{equation*}
where we have set
\begin{align*}
&\widetilde{\ell}_1 :=\begin{pmatrix}
i\, \eta_0 -2\, u_\ell \, \ubeta_1^+ & -i \, \ceta^T & \ubeta_1^+ \end{pmatrix} \, , 
&\widetilde{\ell}_3 :=\begin{pmatrix}
-u_\ell^2 \, |\uceta|^2 & \eta_0 \, \ceta^T & u_\ell \, |\uceta|^2 \end{pmatrix} \, ,\\
&\widetilde{\ell}_2 :=\begin{pmatrix}
-i\, \eta_0 -2\, u_r \, \ubeta_2^+ & i \, \ceta^T & \ubeta_2^+ \end{pmatrix} \, , &
\end{align*}
and
\begin{align*}
&\omega_1 := [\rho] \, [u] \, \underline{\Upsilon} \, 
\dfrac{\ueta_0^2 +u_r^2 \, |\uceta|^2}{\ueta_0^2 +u_\ell^2 \, |\uceta|^2} \, i \, u_\ell \, \ueta_0 \, 
(u_r \, \ua_r -i\, c_r^2 \, \ueta_0) ,& \\
&\omega_3 := [\rho] \, [u]^2 \, \underline{\Upsilon} \, 
\dfrac{\ueta_0^2 -u_\ell \, u_r \, |\uceta|^2}{\ueta_0^2 +u_\ell^2 \, |\uceta|^2} \, (u_r \, \ua_r -i\, c_r^2 \, \ueta_0) ,
&\omega_2 := [\rho] \, [u] \, \underline{\Upsilon} \, i \, u_r \, \ueta_0 \, (u_\ell \, \ua_\ell -i\, c_\ell^2 \, \ueta_0) .
\end{align*}
\end{lemma}

\begin{proof}
We first observe that $\sigma$ is orthogonal to the vectors $H(\underline{v})\, \underline{R}_p^+$ for $p \ge 4$, 
so for $k>0$ the expression of $L(k,z)$ reduces to
\begin{equation*}
L(k,z) =\sum_{p=1}^3 \sigma^* \, H(\underline{v}) \, \underline{R}_p^+ \, \exp (-k\, \ubeta_p^+ \, z) \, 
(\underline{L}_p^+)^* \, .
\end{equation*}
The expression of the products $\sigma^* \, H(\underline{v}) \, \underline{R}_p^+$, $p=1,2,3$ can be found in 
\eqref{prod1}, \eqref{prod2}, \eqref{prod3}, and we use the definitions of the left eigenvectors $\underline{L}_p^+$ 
to derive the expressions given in Lemma \ref{lem2}.
\end{proof}

\noindent We now examine the kernel $q_5$, which, as $q_1$ but unlike $q_2,q_3,q_4$, does not contain any term 
in $p''(\rho_{\ell,r})$.

\begin{lemma}
\label{lem3}
With $Q$ defined in \eqref{defQ}, the kernel $q_5$ in \eqref{defq5} satisfies
\begin{equation*}
q_5(k,k') =\begin{cases}
0 &\text{\rm if $k>0$ and $k'>0$,}\\
\overline{Q} \, \dfrac{k'}{k} &\text{\rm if $k>0$, $k'<0$ and $k+k'>0$.}
\end{cases}
\end{equation*}
\end{lemma}

\begin{proof}
For $k>0$ and $k'>0$, we compute
\begin{equation*}
\breve{\A} (v,\ueta) \left( 
\dfrac{\partial \widehat{r}}{\partial z} (k,z) +\dfrac{\partial \widehat{r}}{\partial z} (k',z) \right) 
=\begin{pmatrix}
i\, \gamma_1 \, (\ubeta_1^-)^2 \, (k\, \exp (k\, \ubeta_1^- \, z) +k'\, \exp (k'\, \ubeta_1^- \, z)) \, 
A^d(v_\ell) \, \underline{r}_1^- \\
i\, \gamma_2 \, (\ubeta_2^-)^2 \, (k\, \exp (k\, \ubeta_2^- \, z) +k'\, \exp (k'\, \ubeta_2^- \, z)) \, 
A^d(v_r) \, \underline{r}_2^- \end{pmatrix} \, .
\end{equation*}
Using the orthogonality properties $\widetilde{\ell}_1 \, A^d(v_\ell) \, \underline{r}_1^- 
=\widetilde{\ell}_3 \, A^d(v_\ell) \, \underline{r}_1^- =\widetilde{\ell}_3 \, A^d(v_r) \, \underline{r}_2^- =0$, 
we get $q_5(k,k')=0$ if $k>0$ and $k'>0$ because the integrand in \eqref{defq5} vanishes.

Let us now consider the case $k>0$, $k'<0$ and $k+k'>0$. From the previous argument, we still find that the term 
$L(k+k',z) \, \breve{\A}(v,\ueta) \, \partial_z \widehat{r}(k,z)$ vanishes. We thus get
\begin{align*}
q_5(k,k') &= -\int_0^{+\infty} L(k+k',z) \, \breve{\A} (v,\ueta) \, \dfrac{\partial \widehat{r}}{\partial z} (k',z) \, \dif z \\
&=-i\, k' \, \int_0^{+\infty} L(k+k',z) \, \begin{pmatrix}
\overline{\gamma_1} \, (\ubeta_1^+)^2 \, \exp (k'\, \ubeta_1^+ \, z) \, A^d(v_\ell) \, \underline{r}_1^+ \\
\overline{\gamma_2} \, (\ubeta_2^+)^2 \, \exp (k'\, \ubeta_2^+ \, z) \, A^d(v_r) \, \underline{r}_2^+ \end{pmatrix} \, \dif z \, .
\end{align*}
We now use the expression of $L(k+k',z)$ in Lemma \ref{lem2}. The expression of $q_5(k,k')$ is simplified by 
recalling the orthogonality property $\widetilde{\ell}_3 \, A^d(v_\ell) \, \underline{r}_1^+ =0$ and we get
\begin{align*}
q_5(k,k') =& -i\, k' \, \int_0^{+\infty} \dfrac{\overline{\gamma_1}}{\gamma_1} \, \omega_1 \, 
\widetilde{\ell}_1 \, A^d(v_\ell) \, \underline{r}_1^+ \, (\ubeta_1^+)^2 \, \exp (-k\, \ubeta_1^+ \, z) \, \dif z \\
&-i\, k' \, \int_0^{+\infty} \dfrac{\overline{\gamma_2}}{\gamma_2} \, \omega_2 \, 
\widetilde{\ell}_2 \, A^d(v_r) \, \underline{r}_2^+ \, (\ubeta_2^+)^2 \, \exp (-k\, \ubeta_2^+ \, z) \, \dif z \, ,\\
=& -i\, \dfrac{k'}{k} \, \left\{ \dfrac{\overline{\gamma_1}}{\gamma_1} \, \omega_1 \, 
\widetilde{\ell}_1 \, A^d(v_\ell) \, \underline{r}_1^+ \, \ubeta_1^+ +\dfrac{\overline{\gamma_2}}{\gamma_2} \, \omega_2 \, 
\widetilde{\ell}_2 \, A^d(v_r) \, \underline{r}_2^+ \, \ubeta_2^+ \right\} \, .
\end{align*}
The conclusion of Lemma \ref{lem3} then follows from the relations
\begin{align*}
&\dfrac{\overline{\gamma_1}}{\gamma_1} =-\dfrac{\overline{\gamma_2}}{\gamma_2}
=-\dfrac{u_\ell \, \ua_r -i\, c_r^2 \, \ueta_0}{u_\ell \, \ua_r +i\, c_r^2 \, \ueta_0} \, ,\\
&\omega_1 \, \widetilde{\ell}_1 \, A^d(v_\ell) \, \underline{r}_1^+ 
=-\omega_2 \, \widetilde{\ell}_2 \, A^d(v_r) \, \underline{r}_2^+ 
=2\, [\rho] \, [u] \, \underline{\Upsilon} \, (\ueta_0^2 +u_r^2 \, |\uceta|^2) \, u_\ell \, u_r \, \ua_\ell \, \ua_r \, .
\end{align*}
\end{proof}

\noindent We immediately get

\begin{corollary}
\label{cor1}
With $Q$ defined in \eqref{defQ}, the kernels $q_1,q_5$ in \eqref{defq1}, \eqref{defq5} satisfy
\begin{equation*}
(q_1+q_5)(k,k') =\begin{cases}
0 &\text{\rm if $k>0$ and $k'>0$,}\\
\overline{Q} \, \left( 1+\dfrac{k'}{k} \right) &\text{\rm if $k>0$, $k'<0$ and $k+k'>0$.}
\end{cases}
\end{equation*}
\end{corollary}

Our goal now is to derive an explicit expression for the kernels $q_2,q_3,q_4$ in \eqref{defq2}, \eqref{defq3}, 
\eqref{defq4}. This is more intricate because these kernels are quadratic with respect to the vector $\widehat{r}(k,z)$ 
and there is more algebra involved to obtain a factorized expression in each region of the $(k,k')$-plane. We begin 
with some preliminary computations that will be useful in Propositions \ref{prop2} and \ref{prop3} below.

\begin{lemma}
\label{lem4}
The coordinates \eqref{defD} of the vector $\sigma$ satisfy
\begin{align*}
\gamma_1 \, \check{D} &= -[\rho] \, [u] \, u_r \, \ueta_0 \, (u_r \, \ua_r -i \, c_r^2 \, \ueta_0) \, ,\\
\gamma_2 \, \check{D} &= [\rho] \, [u] \, u_r \, \ueta_0 \, (u_\ell \, \ua_\ell -i \, c_\ell^2 \, \ueta_0) \, ,\\
\gamma_1 \, D_{d+1} &= -[\rho] \, u_r \, (\ueta_0^2 +u_r^2 \, |\uceta|^2) \, (u_\ell \, \ua_r +i \, c_r^2 \, \ueta_0) \, ,\\
\gamma_2 \, D_{d+1} &= -[\rho] \, u_\ell \, (\ueta_0^2 +u_r^2 \, |\uceta|^2) \, (u_r \, \ua_\ell +i \, c_\ell^2 \, \ueta_0) \, ,\\
\gamma_1 \, D_{d+2} &= [\rho] \, (\ueta_0^2 +u_r^2 \, |\uceta|^2) \, (u_r \, \ua_r +i \, c_r^2 \, \ueta_0) 
-[\rho] \, [u] \, u_r \, |\uceta|^2 \, (u_r \, \ua_r -i \, c_r^2 \, \ueta_0) \, ,\\
\gamma_2 \, D_{d+2} &=  [\rho] \, (\ueta_0^2 +u_r^2 \, |\uceta|^2) \, (u_\ell \, \ua_\ell +i \, c_\ell^2 \, \ueta_0) 
+[\rho] \, [u] \, u_r \, |\uceta|^2 \, (u_\ell \, \ua_\ell -i \, c_\ell^2 \, \ueta_0) \, .
\end{align*}
\end{lemma}

\begin{proof}
From the definition \eqref{defD}, there holds
\begin{equation*}
\check{D} =[u] \, u_r \, \big( \ua_r \, (u_r \, \ua_\ell -i\, c_\ell^2 \, \ueta_0) 
+\ua_\ell \, (u_\ell \, \ua_r -i\, c_r^2 \, \ueta_0) \big) \, ,
\end{equation*}
and we then use the expression \eqref{defgamma} of $\gamma_1,\gamma_2$ to compute $\gamma_1 \, \check{D}$ 
and $\gamma_2 \, \check{D}$.

The expressions of $\gamma_{1,2} \, D_{d+1}$ follow from the observation that $D_{d+1}$ satisfies
\begin{align*}
\ueta_0 \, D_{d+1} 
&=-(\ueta_0^2 +u_r^2 \, |\uceta|^2) \, (u_r \, \ua_\ell -i \, c_\ell^2 \, \ueta_0) \, (u_\ell \, \ua_r +i \, c_r^2 \, \ueta_0) \\
&=(\ueta_0^2 +u_r^2 \, |\uceta|^2) \, (u_r \, \ua_\ell +i \, c_\ell^2 \, \ueta_0) \, (u_\ell \, \ua_r -i \, c_r^2 \, \ueta_0) \, .
\end{align*}
We then use again \eqref{defgamma}.

To compute the product $\gamma_1 \, D_{d+2}$, we recall the relation \eqref{prod2} which we found in the 
proof of Proposition \ref{prop1}. It reads
\begin{equation*}
\ueta_0 \, (D_{d+1} +u_\ell \, D_{d+2}) -u_\ell \, |\uceta|^2 \check{D} =[u] \, \ua_\ell \, (\ueta_0^2 +u_r^2 \, |\uceta|^2) 
\, (u_\ell \, \ua_r -i\, c_r^2 \, \ueta_0) \, .
\end{equation*}
We multiply the latter relation by $\gamma_1$ and use the previous expressions of $\gamma_1 \, \check{D}$ and 
$\gamma_1 \, D_{d+1}$ to obtain that of $\gamma_1 \, D_{d+2}$. The expression of $\gamma_2 \, D_{d+2}$ is 
then easily deduced by using $\gamma_2/\gamma_1 =i \, c_\ell^2 \, \ueta_0^2 /(u_r \, \ua_r)$.
\end{proof}

\noindent We now compute the kernels $q_2,q_3,q_4$ in the region $\{ k>0, k'>0 \}$.

\begin{proposition}
\label{prop2}
Let us define the quantities
\begin{align}
Q_\ell &:= [\rho] \, [u] \, \underline{\Upsilon} \, u_\ell \, u_r \, \dfrac{\ua_r}{\ua_\ell} \, (\ueta_0^2 +u_r^2 \, |\uceta|^2) 
\, (\ueta_0^2 +u_\ell^2 \, |\uceta|^2) \, \gamma_1 \, (i\, \ueta_0 -u_\ell \, \ubeta_1^-) \, ,\notag \\
Q_r &:= [\rho] \, [u] \, \underline{\Upsilon} \, u_\ell \, u_r \, \dfrac{\ua_\ell}{\ua_r} \, (\ueta_0^2 +u_r^2 \, |\uceta|^2)^2 
\, \gamma_2 \, (i\, \ueta_0 +u_r \, \ubeta_2^-) \, ,\label{defQlr} \\
Q_\sharp &:=2\, [\rho] \, \underline{\Upsilon} \, (\ueta_0^2 +u_r^2 \, |\uceta|^2) \, (\ueta_0^2 +u_\ell \, u_r \, |\uceta|^2) \, 
i \, c_\ell^2 \, c_r^2 \, \ueta_0 \, \left( \dfrac{c_r^2 \, \gamma_2}{\rho_r \, u_r} -\dfrac{c_\ell^2 \, \gamma_1}{\rho_\ell \, u_\ell} 
\right) \, .\notag 
\end{align}
Then the kernels $q_2,q_3,q_4$ defined in \eqref{defq2}, \eqref{defq3} and \eqref{defq4} satisfy
\begin{equation*}
(q_2+q_3+q_4)(k,k') =\left( \dfrac{p''(\rho_\ell)}{2} +\dfrac{c_\ell^2}{\rho_\ell} \right) \, Q_\ell 
+\left( \dfrac{p''(\rho_r)}{2} +\dfrac{c_r^2}{\rho_r} \right) \, Q_r +Q_\sharp \, ,
\end{equation*}
for all $k>0$ and $k'>0$.
\end{proposition}

\begin{proof}
We first recall the expressions of the second differentials that are involved in the kernels $q_2,q_3,q_4$, see 
\cite[page 1480]{BenzoniRosini}:
\begin{equation}
\label{hessienne1}
\sum_{k=1}^{d-1} \uceta_k \, \dif^2 f^k (v_{\ell,r}) \cdot (v,v) =p''(\rho_{\ell,r}) \, \begin{pmatrix}
0 \\
\rho^2 \, \uceta \\
0 \end{pmatrix} +\dfrac{2}{\rho_{\ell,r}} \, \begin{pmatrix}
0 \\
\uceta \cdot \check{j} \, \check{j} \\
\uceta \cdot \check{j} \, (j_d -u_{\ell,r} \, \rho) \end{pmatrix} \, ,
\end{equation}
\begin{multline}
\label{hessienne2}
\dif^2 \tilde{f}^d (v_{\ell,r}) \cdot (v,v) =\begin{pmatrix}
\dif^2 f^d (v_{\ell,r}) \cdot (v,v) \\
\dif^2 g^d (v_{\ell,r}) \cdot (v,v) \end{pmatrix} =p'' (\rho_{\ell,r}) \, \begin{pmatrix}
0 \\
0 \\
\rho^2 \\
u_{\ell,r} \, \rho^2 \end{pmatrix} +\dfrac{2}{\rho_{\ell,r}} \, \begin{pmatrix}
0 \\
(j_d -u_{\ell,r} \, \rho) \, \check{j} \\
(j_d -u_{\ell,r} \, \rho)^2 \\
0 \end{pmatrix} \\
+\dfrac{1}{\rho_{\ell,r}} \, \begin{pmatrix}
0 \\
0 \\
0 \\
3\, u_{\ell,r} \, (j_d -u_{\ell,r} \, \rho)^2 -u_{\ell,r} \, c_{\ell,r}^2 \, \rho^2 
+2\, c_{\ell,r}^2 \, \rho \, j_d +u_{\ell,r} \, \check{j} \cdot \check{j} \end{pmatrix} \, .
\end{multline}
The proof of Proposition \ref{prop2} splits in several steps. We assume from now on that $k$ and $k'$ 
are both positive and we wish to compute the expression of the kernel $q_2+q_3+q_4$.
\bigskip

$\bullet$ \underline{Step 1: computation of the $p''(\rho_\ell)$ factor.} In this first step, we collect all the terms 
that involve $p''(\rho_\ell)$ in $q_2+q_3+q_4$. The contribution of the kernel $q_2$ equals
\begin{multline}
\label{prop2-1}
\underline{\Upsilon} \, (D_{d+1} +u_\ell \, D_{d+2}) \, \gamma_1^2 \, (i\, \ueta_0 -u_\ell \, \ubeta_1^-)^2 \\
=-[\rho] \, [u] \, \underline{\Upsilon} \, \gamma_1 \, (i\, \ueta_0 -u_\ell \, \ubeta_1^-)^2 \, \Big\{ 
(\ueta_0^2 +u_r^2 \, |\uceta|^2) \, i \, c_r^2 \, \ueta_0 +(u_r \, \ua_r -i \, c_r^2 \, \ueta_0) \, u_\ell \, u_r \, 
|\uceta|^2 \Big\} \, ,
\end{multline}
where we have used Lemma \ref{lem4} to compute the product $(D_{d+1} +u_\ell \, D_{d+2}) \, \gamma_1$.

The contribution of the kernel $q_3$ equals
\begin{align*}
& i\, (k+k') \, \int_0^{+\infty} \dfrac{\omega_1}{\gamma_1} \, (-i \, \uceta^T) \, {\rm e}^{-(k+k') \, \ubeta_1^+ \, z} \, 
\gamma_1^2 \, (i\, \ueta_0 -u_\ell \, \ubeta_1^-)^2 \, \uceta \, {\rm e}^{(k+k') \, \ubeta_1^- \, z} \, \dif z \\
+\, & i\, (k+k') \, \int_0^{+\infty} \dfrac{\omega_3}{\gamma_1} \, (\ueta_0 \, \uceta^T) \, {\rm e}^{-(k+k') \, \ubeta_3^+ \, z} \, 
\gamma_1^2 \, (i\, \ueta_0 -u_\ell \, \ubeta_1^-)^2 \, \uceta \, {\rm e}^{(k+k') \, \ubeta_1^- \, z} \, \dif z \\
=\, & \gamma_1 \, (i\, \ueta_0 -u_\ell \, \ubeta_1^-)^2 \, \left\{ \omega_1 \, \dfrac{|\uceta|^2}{\ubeta_1^+ -\ubeta_1^-} 
+\omega_3 \, |\uceta|^2 \, \dfrac{i\, \ueta_0}{\ubeta_3^+ -\ubeta_1^-} \right\} \, .
\end{align*}
Similarly, the contribution of the kernel $q_4$ reads
\begin{align*}
& -\int_0^{+\infty} \dfrac{\omega_1}{\gamma_1} \, \ubeta_1^+ \, {\rm e}^{-(k+k') \, \ubeta_1^+ \, z} \, \gamma_1^2 \, 
(i\, \ueta_0 -u_\ell \, \ubeta_1^-)^2 \, (k+k') \, \ubeta_1^- \, {\rm e}^{(k+k') \, \ubeta_1^- \, z} \, \dif z \\
& -\int_0^{+\infty} \dfrac{\omega_3}{\gamma_1} \, u_\ell \, |\uceta|^2 \, {\rm e}^{-(k+k') \, \ubeta_3^+ \, z} \, \gamma_1^2 \, 
(i\, \ueta_0 -u_\ell \, \ubeta_1^-)^2 \, (k+k') \, \ubeta_1^- \, {\rm e}^{(k+k') \, \ubeta_1^- \, z} \, \dif z \\
=\, & -\gamma_1 \, (i\, \ueta_0 -u_\ell \, \ubeta_1^-)^2 \, \left\{ 
\omega_1 \, \dfrac{\ubeta_1^+ \, \ubeta_1^-}{\ubeta_1^+ -\ubeta_1^-} 
+\omega_3 \, |\uceta|^2 \, \dfrac{u_\ell \, \ubeta_1^-}{\ubeta_3^+ -\ubeta_1^-} \right\} \, .
\end{align*}
Adding the contributions of $q_3$ and $q_4$ gives the term
\begin{equation}
\label{prop2-2}
\gamma_1 \, (i\, \ueta_0 -u_\ell \, \ubeta_1^-)^2 \, \left\{ -\dfrac{1}{2\, \ua_\ell} \, \left( \dfrac{\ua_\ell^2}{c_\ell^2} 
+c_\ell^2 \, |\uceta|^2 \right) \, \omega_1 \, +u_\ell \, |\uceta|^2 \, \omega_3 \right\} \, .
\end{equation}
We now use the definitions of $\omega_1$ and $\omega_3$, see Lemma \ref{lem2}, and add the contributions 
in \eqref{prop2-1} and \eqref{prop2-2} in order to obtain the $p''(\rho_\ell)$ term in $q_2+q_3+q_4$. We first obtain 
that the sum of the right hand side of \eqref{prop2-1} and the expression in \eqref{prop2-2} equals
\begin{equation*}
-[\rho] \, [u] \, \underline{\Upsilon} \, \gamma_1 \, (i\, \ueta_0 -u_\ell \, \ubeta_1^-)^2 \, 
\dfrac{\ueta_0^2 +u_r^2 \, |\uceta|^2}{2\, c_\ell^2 \, (\ueta_0^2 +u_\ell^2 \, |\uceta|^2)} \, 
(u_\ell \, \ua_\ell +i\, c_\ell^2 \, \ueta_0) \, \left( c_r^2 \, \ueta_0^2 +\dfrac{u_r \, \ua_r}{u_\ell \, \ua_\ell} \, 
u_\ell^2 \, c_\ell^2 \, |\uceta|^2 \right) \, .
\end{equation*}
This last expression is simplified a little further by observing that we have
\begin{equation*}
(i\, \ueta_0 -u_\ell \, \ubeta_1^-) \, (u_\ell \, \ua_\ell +i\, c_\ell^2 \, \ueta_0) =\dfrac{-1}{c_\ell^2 -u_\ell^2} 
\, (u_\ell^2 \, \ua_\ell^2 +c_\ell^4 \, \ueta_0^2) =-c_\ell^2 \, (\ueta_0^2 +u_\ell^2 \, |\uceta|^2) \, ,
\end{equation*}
and
\begin{equation*}
c_r^2 \, \ueta_0^2 +\dfrac{u_r \, \ua_r}{u_\ell \, \ua_\ell} \, u_\ell^2 \, c_\ell^2 \, |\uceta|^2 
=u_\ell \, u_r \, \dfrac{\ua_r}{\ua_\ell} \, \left( c_\ell^2 \, |\uceta|^2 -\dfrac{\ua_\ell^2}{c_\ell^2} \right) 
=u_\ell \, u_r \, \dfrac{\ua_r}{\ua_\ell} \, (\ueta_0^2 +u_\ell^2 \, |\uceta|^2) \, .
\end{equation*}
Eventually, we find that the sum of all the terms that involve $p''(\rho_\ell)$ factorizes as $Q_\ell/2$ where 
$Q_\ell$ is defined in \eqref{defQlr}.
\bigskip

$\bullet$ \underline{Step 2: computation of the $p''(\rho_r)$ factor.} We follow the same strategy as in the first 
step and compute the contribution that involves $p''(\rho_r)$ in each kernel. The contribution of the kernel 
$q_2$ equals
\begin{equation}
\label{prop2-3}
-\underline{\Upsilon} \, (D_{d+1} +u_r \, D_{d+2}) \, \gamma_2^2 \, (i\, \ueta_0 +u_r \, \ubeta_2^-)^2 
=[\rho] \, [u] \, \underline{\Upsilon} \, \gamma_2 \, (i\, \ueta_0 +u_r \, \ubeta_2^-)^2 \, \dfrac{u_\ell \, \ua_\ell}{c_r^2} 
\, \Big\{ i \, \ueta_0 \, u_r \, \ua_r -u_r^2 \, c_r^2 \, |\uceta|^2 \Big\} \, ,
\end{equation}
where we have used Lemma \ref{lem4} to simplify $(D_{d+1} +u_r \, D_{d+2}) \, \gamma_2$. The contribution 
of the kernel $q_3$ equals
\begin{multline*}
i\, (k+k') \, \int_0^{+\infty} \dfrac{\omega_2}{\gamma_2} \, (i \, \uceta^T) \, {\rm e}^{-(k+k') \, \ubeta_2^+ \, z} \, 
\gamma_2^2 \, (i\, \ueta_0 +u_r \, \ubeta_2^-)^2 \, \uceta \, {\rm e}^{(k+k') \, \ubeta_2^- \, z} \, \dif z \\
=-\gamma_2 \, (i\, \ueta_0 +u_r \, \ubeta_2^-)^2 \, \omega_2 \, \dfrac{|\uceta|^2}{\ubeta_2^+ -\ubeta_2^-} \, ,
\end{multline*}
and the contribution of the kernel $q_4$ reads
\begin{multline*}
\int_0^{+\infty} \dfrac{\omega_2}{\gamma_2} \, \ubeta_2^+ \, {\rm e}^{-(k+k') \, \ubeta_2^+ \, z} \, \gamma_2^2 \, 
(i\, \ueta_0 +u_r \, \ubeta_2^-)^2 \, (k+k') \, \ubeta_2^- \, {\rm e}^{(k+k') \, \ubeta_2^- \, z} \, \dif z \\
=\gamma_2 \, (i\, \ueta_0 +u_r \, \ubeta_2^-)^2 \, \omega_2 \, 
\dfrac{\ubeta_2^+ \, \ubeta_2^-}{\ubeta_2^+ -\ubeta_2^-} \, .
\end{multline*}
Adding the contributions of $q_3$ and $q_4$ gives the term
\begin{equation*}
-\gamma_2 \, (i\, \ueta_0 +u_r \, \ubeta_2^-)^2 \, \dfrac{\omega_2}{2\, \ua_r} \, \left( \dfrac{\ua_r^2}{c_r^2} 
+c_r^2 \, |\uceta|^2 \right) \, .
\end{equation*}
When we add the latter term with the expression in \eqref{prop2-3}, we obtain
\begin{equation*}
-[\rho] \, [u] \, \underline{\Upsilon} \, \gamma_2 \, (i\, \ueta_0 +u_r \, \ubeta_2^-)^2 \, (\ueta_0^2 +u_r^2 \, |\uceta|^2) 
\, \dfrac{i\, u_r \, \ueta_0}{2\, \ua_r} \, (u_\ell \, \ua_\ell +i\, c_\ell^2 \, \ueta_0) \, ,
\end{equation*}
and this quantity is further simplified by using the relation
\begin{equation*}
(i\, \ueta_0 +u_r \, \ubeta_2^-) \, i \, \ueta_0 \, (u_\ell \, \ua_\ell +i\, c_\ell^2 \, \ueta_0) =-u_\ell \, \ua_\ell 
\, \dfrac{u_r^2 \, \ua_r^2 +c_r^4 \, \ueta_0^2}{c_r^2 \, (c_r^2 -u_r^2)} 
=-u_\ell \, \ua_\ell \, (\ueta_0^2 +u_r^2 \, |\uceta|^2) \, .
\end{equation*}
Eventually, we find that the sum of all the terms that involve $p''(\rho_r)$ factorizes as $Q_r/2$.
\bigskip

$\bullet$ \underline{Step 3: computation of the remaining terms.} In order to prove Proposition \ref{prop2}, we can 
assume from now on, and without loss of generality that $p''(\rho_\ell) =p''(\rho_r) =0$ in \eqref{hessienne1} and 
\eqref{hessienne2}. With this simplification, we compute
\begin{align}
\sum_{k=1}^{d-1} \uceta_k \, \dif^2 f^k (v_\ell) \cdot (\underline{r}_1^-,\underline{r}_1^-) 
&=-\dfrac{2\, i\, c_\ell^4 \, |\uceta|^2}{\rho_\ell} \, \begin{pmatrix}
0 \\
-i \, \uceta \\
\ubeta_1^- \end{pmatrix} \, ,\label{prop2-4} \\
\dif^2 \tilde{f}^d (v_\ell) \cdot (\underline{r}_1^-,\underline{r}_1^-) &=\dfrac{2\, c_\ell^2}{\rho_\ell} \, 
\begin{pmatrix}
0 \\
-i \, c_\ell^2 \, \ubeta_1^- \uceta \\
c_\ell^2 \, |\uceta|^2 +(i\, \ueta_0 -u_\ell \, \ubeta_1^-)^2 \\
i \, c_\ell^2 \, \ueta_0 \, \ubeta_1^- +u_\ell \, (i\, \ueta_0 -u_\ell \, \ubeta_1^-)^2 \end{pmatrix} \, ,\label{prop2-5} \\
\sum_{k=1}^{d-1} \uceta_k \, \dif^2 f^k (v_r) \cdot (\underline{r}_2^-,\underline{r}_2^-) 
&=\dfrac{2\, i\, c_r^4 \, |\uceta|^2}{\rho_r} \, \begin{pmatrix}
0 \\
i \, \uceta \\
\ubeta_2^- \end{pmatrix} \, ,\label{prop2-6} \\
\dif^2 \tilde{f}^d (v_r) \cdot (\underline{r}_2^-,\underline{r}_2^-) &=\dfrac{2\, c_r^2}{\rho_r} \, 
\begin{pmatrix}
0 \\
i \, c_r^2 \, \ubeta_2^- \uceta \\
c_r^2 \, |\uceta|^2 +(i\, \ueta_0 +u_r \, \ubeta_2^-)^2 \\
-i \, c_r^2 \, \ueta_0 \, \ubeta_2^- +u_r \, (i\, \ueta_0 +u_r \, \ubeta_2^-)^2 \end{pmatrix} \, ,\label{prop2-7}
\end{align}
where in \eqref{prop2-5} and \eqref{prop2-7}, we have used the relations
\begin{equation*}
c_\ell^2 \, (\ubeta_1^-)^2 =c_\ell^2 \, |\uceta|^2 +(i\, \ueta_0 -u_\ell \, \ubeta_1^-)^2 \, ,\quad 
c_r^2 \, (\ubeta_2^-)^2 =c_r^2 \, |\uceta|^2 +(i\, \ueta_0 +u_r \, \ubeta_2^-)^2 \, .
\end{equation*}

With these expressions, let us look first at the kernel $q_2$ in \eqref{defq2}. Using \eqref{prop2-5}, we compute
\begin{multline*}
\sigma^* \, \dif^2 \tilde{f}^d (v_\ell) \cdot (\underline{r}_1^-,\underline{r}_1^-) \, \gamma_1^2 
=\dfrac{2\, c_\ell^2}{\rho_\ell} \, \underline{\Upsilon} \, \gamma_1 \, \Big\{ i \, c_\ell^2 \, \ubeta_1^- \, 
(\ueta_0 \, \gamma_1 \, D_{d+2} -|\uceta|^2 \, \gamma_1 \, \check{D}) +c_\ell^2 \, |\uceta|^2 \, \gamma_1 \, D_{d+1} \\
+(i\, \ueta_0 -u_\ell \, \ubeta_1^-)^2 \, \gamma_1 \, (D_{d+1} +u_\ell \, D_{d+2}) \Big\} \, ,
\end{multline*}
and Lemma \ref{lem4} turns this expression into
\begin{multline}
\label{prop2-8}
\sigma^* \, \dif^2 \tilde{f}^d (v_\ell) \cdot (\underline{r}_1^-,\underline{r}_1^-) \, \gamma_1^2 \\
=\dfrac{2\, c_\ell^2}{\rho_\ell} \, [\rho] \, \underline{\Upsilon} \, (\ueta_0^2 +u_r^2 \, |\uceta|^2) \, \gamma_1 \, 
\Big\{ u_r \, \ua_r \, (u_\ell \, \ua_\ell +i\, c_\ell^2 \, \ueta_0) \, \ubeta_1^- 
-u_r \, c_\ell^2 \, |\uceta|^2 \, (u_\ell \, \ua_r +i\, c_r^2 \, \ueta_0) \Big\} \\
-\dfrac{2\, c_\ell^2}{\rho_\ell} \, [\rho] \, [u] \, \underline{\Upsilon} \, \gamma_1 \, (i\, \ueta_0 -u_\ell \, \ubeta_1^-)^2 \, 
\Big\{ (\ueta_0^2 +u_r^2 \, |\uceta|^2) \, i\, c_r^2 \, \ueta_0 +u_\ell \, u_r \, |\uceta|^2 \, (u_r \, \ua_r -i\, c_r^2 \, \ueta_0) 
\Big\} \, .
\end{multline}
Similarly, we derive the relation
\begin{multline}
\label{prop2-9}
\sigma^* \, \dif^2 \tilde{f}^d (v_r) \cdot (\underline{r}_2^-,\underline{r}_2^-) \, \gamma_2^2 \\
=-\dfrac{2\, c_r^2}{\rho_r} \, [\rho] \, \underline{\Upsilon} \, (\ueta_0^2 +u_r^2 \, |\uceta|^2) \, \gamma_2 \, 
\Big\{ u_\ell \, \ua_\ell \, (u_r \, \ua_r +i\, c_r^2 \, \ueta_0) \, \ubeta_2^- 
+u_\ell \, c_r^2 \, |\uceta|^2 \, (u_r \, \ua_\ell +i\, c_\ell^2 \, \ueta_0) \Big\} \\
+\dfrac{2\, c_r^2}{\rho_r} \, [\rho] \, [u] \, \underline{\Upsilon} \, \gamma_2 \, (i\, \ueta_0 +u_r \, \ubeta_2^-)^2 \, 
\Big\{ (\ueta_0^2 +u_r^2 \, |\uceta|^2) \, i\, c_\ell^2 \, \ueta_0 +u_r^2 \, |\uceta|^2 \, (u_\ell \, \ua_\ell -i\, c_\ell^2 \, \ueta_0) 
\Big\} \, .
\end{multline}
\bigskip

The kernels $q_3,q_4$ both read as a sum of two contributions, one from the system ahead of the phase boundary, 
and one from the system behind the phase boundary. We compute each of these contributions separately in order to 
combine them with either \eqref{prop2-8} or \eqref{prop2-9}. Using \eqref{prop2-4}, the `left' contribution of $q_3$ 
equals
\begin{align}
& i\, (k+k') \, \int_0^{+\infty} \dfrac{\omega_1}{\gamma_1} \, \widetilde{\ell}_1 \, {\rm e}^{-(k+k') \, \ubeta_1^+ \, z} \, 
\gamma_1^2 \, \dfrac{-2\, i\, c_\ell^4 \, |\uceta|^2}{\rho_\ell} \, \begin{pmatrix}
0 \\
-i \, \uceta \\
\ubeta_1^- \end{pmatrix} \, {\rm e}^{(k+k') \, \ubeta_1^- \, z} \, \dif z \notag \\
+\, & i\, (k+k') \, \int_0^{+\infty} \dfrac{\omega_3}{\gamma_1} \, \widetilde{\ell}_3 \, {\rm e}^{-(k+k') \, \ubeta_3^+ \, z} \, 
\gamma_1^2 \, \dfrac{-2\, i\, c_\ell^4 \, |\uceta|^2}{\rho_\ell} \, \begin{pmatrix}
0 \\
-i \, \uceta \\
\ubeta_1^- \end{pmatrix} \, {\rm e}^{(k+k') \, \ubeta_1^- \, z} \, \dif z \notag \\
=\, & \dfrac{2\, c_\ell^4}{\rho_\ell} \, \gamma_1 \, \omega_1 \, \dfrac{|\uceta|^2}{\ubeta_1^+ -\ubeta_1^-} \, 
\widetilde{\ell}_1 \, \begin{pmatrix}
0 \\
-i \, \uceta \\
\ubeta_1^- \end{pmatrix} +\dfrac{2\, c_\ell^4}{\rho_\ell} \, \gamma_1 \, \omega_3 \, 
\dfrac{|\uceta|^2}{\ubeta_3^+ -\ubeta_1^-} \, \widetilde{\ell}_3 \, \begin{pmatrix}
0 \\
-i \, \uceta \\
\ubeta_1^- \end{pmatrix} \, ,\label{prop2-10}
\end{align}
and, similarly (using now \eqref{prop2-6} rather than \eqref{prop2-4}), the `right' contribution of $q_3$ equals
\begin{multline}
\label{prop2-11}
i\, (k+k') \, \int_0^{+\infty} \dfrac{\omega_2}{\gamma_2} \, \widetilde{\ell}_2 \, {\rm e}^{-(k+k') \, \ubeta_2^+ \, z} \, 
\gamma_2^2 \, \dfrac{2\, i\, c_r^4 \, |\uceta|^2}{\rho_r} \, \begin{pmatrix}
0 \\
i \, \uceta \\
\ubeta_2^- \end{pmatrix} \, {\rm e}^{(k+k') \, \ubeta_2^- \, z} \, \dif z \\
=-\dfrac{2\, c_r^4}{\rho_r} \, \gamma_2 \, \omega_2 \, \dfrac{|\uceta|^2}{\ubeta_2^+ -\ubeta_2^-} \, 
\widetilde{\ell}_2 \, \begin{pmatrix}
0 \\
i \, \uceta \\
\ubeta_2^- \end{pmatrix} \, .
\end{multline}

The kernel $q_4$ is computed by first observing that the vectors $\dif^2 f^d (v_\ell) \cdot 
(\underline{r}_1^-,\underline{r}_1^-)$, and $\dif^2 f^d (v_r) \cdot (\underline{r}_2^-,\underline{r}_2^-)$ are 
obtained by retaining only the three first coordinates in \eqref{prop2-5} and \eqref{prop2-7}:
\begin{equation*}
\dif^2 f^d (v_\ell) \cdot (\underline{r}_1^-,\underline{r}_1^-) =\dfrac{2\, c_\ell^4 \, \ubeta_1^-}{\rho_\ell} \, 
\begin{pmatrix}
0 \\
-i \, \uceta \\
\ubeta_1^- \end{pmatrix} \, ,\quad \dif^2 f^d (v_r) \cdot (\underline{r}_2^-,\underline{r}_2^-) 
=\dfrac{2\, c_r^4 \, \ubeta_2^-}{\rho_r} \, \begin{pmatrix}
0 \\
i \, \uceta \\
\ubeta_2^- \end{pmatrix} \, .
\end{equation*}
Consequently, the `left' contribution of $q_4$ equals
\begin{align}
& -\int_0^{+\infty} \dfrac{\omega_1}{\gamma_1} \, \widetilde{\ell}_1 \, {\rm e}^{-(k+k') \, \ubeta_1^+ \, z} \, \gamma_1^2 \, 
\dfrac{2\, c_\ell^4 \, \ubeta_1^-}{\rho_\ell} \, 
\begin{pmatrix}
0 \\
-i \, \uceta \\
\ubeta_1^- \end{pmatrix} \, (k+k') \, \ubeta_1^- \, {\rm e}^{(k+k') \, \ubeta_1^- \, z} \, \dif z \notag \\
& -\int_0^{+\infty} \dfrac{\omega_3}{\gamma_1} \, \widetilde{\ell}_1 \, {\rm e}^{-(k+k') \, \ubeta_3^+ \, z} \, \gamma_1^2 \, 
\dfrac{2\, c_\ell^4 \, \ubeta_1^-}{\rho_\ell} \, 
\begin{pmatrix}
0 \\
-i \, \uceta \\
\ubeta_1^- \end{pmatrix} \, (k+k') \, \ubeta_1^- \, {\rm e}^{(k+k') \, \ubeta_1^- \, z} \, \dif z \notag \\
=\, & -\dfrac{2\, c_\ell^4}{\rho_\ell} \, \gamma_1 \, \omega_1 \, \dfrac{(\ubeta_1^-)^2}{\ubeta_1^+ -\ubeta_1^-} \, 
\widetilde{\ell}_1 \, \begin{pmatrix}
0 \\
-i \, \uceta \\
\ubeta_1^- \end{pmatrix} -\dfrac{2\, c_\ell^4}{\rho_\ell} \, \gamma_1 \, \omega_3 \, 
\dfrac{(\ubeta_1^-)^2}{\ubeta_3^+ -\ubeta_1^-} \, \widetilde{\ell}_3 \, \begin{pmatrix}
0 \\
-i \, \uceta \\
\ubeta_1^- \end{pmatrix} \, ,\label{prop2-12}
\end{align}
and the `right' contribution of $q_4$ equals
\begin{equation}
\label{prop2-13}
\dfrac{2\, c_r^4}{\rho_r} \, \gamma_2 \, \omega_2 \, \dfrac{(\ubeta_2^-)^2}{\ubeta_2^+ -\ubeta_2^-} \, 
\widetilde{\ell}_2 \, \begin{pmatrix}
0 \\
i \, \uceta \\
\ubeta_2^- \end{pmatrix} \, .
\end{equation}

The `left' contribution of $q_3+q_4$ is obtained by adding the expressions in \eqref{prop2-10} and \eqref{prop2-12}, 
which gives
\begin{align}
-\dfrac{2\, c_\ell^4}{\rho_\ell} \, &\gamma_1 \, \omega_1 \, \dfrac{(\ubeta_1^-)^2 -|\uceta|^2}{\ubeta_1^+ -\ubeta_1^-} \, 
\widetilde{\ell}_1 \, \begin{pmatrix}
0 \\
-i \, \uceta \\
\ubeta_1^- \end{pmatrix} -\dfrac{2\, c_\ell^4}{\rho_\ell} \, \gamma_1 \, \omega_3 \, 
\dfrac{(\ubeta_1^-)^2 -|\uceta|^2}{\ubeta_3^+ -\ubeta_1^-} \, \widetilde{\ell}_3 \, \begin{pmatrix}
0 \\
-i \, \uceta \\
\ubeta_1^- \end{pmatrix} \notag \\
&=-\dfrac{2\, c_\ell^2}{\rho_\ell} \, \gamma_1 \, (i\, \ueta_0 -u_\ell \, \ubeta_1^-)^2 \, \left\{ 
\omega_1 \, \widetilde{\ell}_1 \, \begin{pmatrix}
0 \\
-i \, \uceta \\
\ubeta_1^- \end{pmatrix} +\omega_3 \, \widetilde{\ell}_3 \, \begin{pmatrix}
0 \\
-i \, \uceta \\
\ubeta_1^- \end{pmatrix} \right\} \notag \\
&=\dfrac{c_\ell^2}{\rho_\ell \, \ua_\ell} \, \gamma_1 \, (i\, \ueta_0 -u_\ell \, \ubeta_1^-)^2 \, \Big\{ 
(\ueta_0^2 +u_\ell^2 \, |\uceta|^2 -2\, c_\ell^2 \, |\uceta|^2) \, \omega_1 +2\, u_\ell \, \ua_\ell \, |\uceta|^2 
\, \omega_3 \Big\} \, .\label{prop2-14}
\end{align}
The `right' contribution of $q_3+q_4$ is obtained by adding the expressions in \eqref{prop2-11} and \eqref{prop2-13}, 
which gives
\begin{multline}
\label{prop2-15}
\dfrac{2\, c_r^4}{\rho_r} \, \gamma_2 \, \omega_2 \, \dfrac{(\ubeta_2^-)^2 -|\uceta|^2}{\ubeta_2^+ -\ubeta_2^-} \, 
\widetilde{\ell}_2 \, \begin{pmatrix}
0 \\
i \, \uceta \\
\ubeta_2^- \end{pmatrix} =\dfrac{2\, c_r^2}{\rho_r} \, \gamma_2 \, \omega_2 \, (i\, \ueta_0 +u_r \, \ubeta_2^-)^2 
\, \dfrac{\ubeta_2^+ \, \ubeta_2^- -|\uceta|^2}{\ubeta_2^+ -\ubeta_2^-} \\
=\dfrac{c_r^2}{\rho_r \, \ua_r} \, \gamma_2 \, (i\, \ueta_0 +u_r \, \ubeta_2^-)^2 \, (\ueta_0^2 +u_r^2 \, |\uceta|^2 
-2\, c_r^2 \, |\uceta|^2) \, \omega_2 \, .
\end{multline}

We can now compute the `left' contribution of the full kernel $q_2+q_3+q_4$ by combining the expression in 
\eqref{prop2-8} with the one in \eqref{prop2-14}. We use here the expressions of $\omega_1$ and $\omega_3$ 
given in Lemma \ref{lem2}. The sum of \eqref{prop2-8} and \eqref{prop2-14} reads
\begin{align*}
A_\ell :=&\dfrac{2\, c_\ell^2}{\rho_\ell} \, [\rho] \, \underline{\Upsilon} \, (\ueta_0^2 +u_r^2 \, |\uceta|^2) \, \gamma_1 \, 
\Big\{ u_r \, \ua_r \, (u_\ell \, \ua_\ell +i\, c_\ell^2 \, \ueta_0) \, \ubeta_1^- 
-u_r \, c_\ell^2 \, |\uceta|^2 \, (u_\ell \, \ua_r +i\, c_r^2 \, \ueta_0) \Big\} \\
&-\dfrac{2\, c_\ell^2}{\rho_\ell} \, [\rho] \, [u] \, \underline{\Upsilon} \, \gamma_1 \, (i\, \ueta_0 -u_\ell \, \ubeta_1^-)^2 \, 
\Big\{ (\ueta_0^2 +u_r^2 \, |\uceta|^2) \, i\, c_r^2 \, \ueta_0 +u_\ell \, u_r \, |\uceta|^2 \, (u_r \, \ua_r -i\, c_r^2 \, \ueta_0) 
\Big\} \\
&+\dfrac{c_\ell^2}{\rho_\ell \, \ua_\ell} \, [\rho] \, [u] \, \underline{\Upsilon} \, \gamma_1 \, (i\, \ueta_0 -u_\ell \, \ubeta_1^-)^2 
\, \left\{ 2\, u_\ell \, \ua_\ell \, |\uceta|^2 \, [u] \, \dfrac{\ueta_0^2 -u_\ell \, u_r \, |\uceta|^2}{\ueta_0^2 +u_\ell^2 \, |\uceta|^2} 
\, (u_r \, \ua_r -i\, c_r^2 \, \ueta_0) \right. \\
&\left. +(\ueta_0^2 +u_\ell^2 \, |\uceta|^2 -2\, c_\ell^2 \, |\uceta|^2) \, i\, u_\ell \, \ueta_0 \, (u_r \, \ua_r -i\, c_r^2 \, \ueta_0) 
\, \dfrac{\ueta_0^2 +u_r^2 \, |\uceta|^2}{\ueta_0^2 +u_\ell^2 \, |\uceta|^2} \right\} \, .
\end{align*}
The three last rows in the definition of $A_\ell$ are factorized after a little bit of algebra, and we get
\begin{align*}
A_\ell =&\dfrac{2\, c_\ell^2}{\rho_\ell} \, [\rho] \, \underline{\Upsilon} \, (\ueta_0^2 +u_r^2 \, |\uceta|^2) \, \gamma_1 \, 
\Big\{ u_r \, \ua_r \, (u_\ell \, \ua_\ell +i\, c_\ell^2 \, \ueta_0) \, \ubeta_1^- 
-u_r \, c_\ell^2 \, |\uceta|^2 \, (u_\ell \, \ua_r +i\, c_r^2 \, \ueta_0) \Big\} \\
&-\dfrac{1}{\rho_\ell} \, [\rho] \, [u] \, \underline{\Upsilon} \, \gamma_1 \, (i\, \ueta_0 -u_\ell \, \ubeta_1^-)^2 
\, u_\ell \, u_r \, \dfrac{\ua_r}{\ua_\ell} \,  (u_\ell \, \ua_\ell +i\, c_\ell^2 \, \ueta_0) \, (\ueta_0^2 +u_r^2 \, |\uceta|^2) \, .
\end{align*}
We have thus shown that the `left' contribution $A_\ell$ of $q_2+q_3+q_4$ reads
\begin{equation}
\label{finalAl}
A_\ell =\dfrac{2\, c_\ell^2}{\rho_\ell} \, [\rho] \, \underline{\Upsilon} \, (\ueta_0^2 +u_r^2 \, |\uceta|^2) \, \gamma_1 \, 
\Big\{ u_r \, \ua_r \, (u_\ell \, \ua_\ell +i\, c_\ell^2 \, \ueta_0) \, \ubeta_1^- 
-u_r \, c_\ell^2 \, |\uceta|^2 \, (u_\ell \, \ua_r +i\, c_r^2 \, \ueta_0) \Big\} +\dfrac{c_\ell^2}{\rho_\ell} \, Q_\ell \, .
\end{equation}
\bigskip

The `right' contribution of $q_2+q_3+q_4$ by combining the expression in \eqref{prop2-9} (with a minus sign, recall 
the definition \eqref{defq2} of the kernel $q_2$) with the one in \eqref{prop2-15}:
\begin{align*}
A_r :=&\dfrac{2\, c_r^2}{\rho_r} \, [\rho] \, \underline{\Upsilon} \, (\ueta_0^2 +u_r^2 \, |\uceta|^2) \, \gamma_2 \, 
\Big\{ u_\ell \, \ua_\ell \, (u_r \, \ua_r +i\, c_r^2 \, \ueta_0) \, \ubeta_2^- 
+u_\ell \, c_r^2 \, |\uceta|^2 \, (u_r \, \ua_\ell +i\, c_\ell^2 \, \ueta_0) \Big\} \\
&-\dfrac{2\, c_r^2}{\rho_r} \, [\rho] \, [u] \, \underline{\Upsilon} \, \gamma_2 \, (i\, \ueta_0 +u_r \, \ubeta_2^-)^2 \, 
\Big\{ (\ueta_0^2 +u_r^2 \, |\uceta|^2) \, i\, c_\ell^2 \, \ueta_0 +u_r^2 \, |\uceta|^2 \, (u_\ell \, \ua_\ell -i\, c_\ell^2 \, \ueta_0) 
\Big\} \\
&+\dfrac{c_r^2}{\rho_r \, \ua_r} \, \gamma_2 \, (i\, \ueta_0 +u_r \, \ubeta_2^-)^2 \, (\ueta_0^2 +u_r^2 \, |\uceta|^2 
-2\, c_r^2 \, |\uceta|^2) \, \omega_2 \, .
\end{align*}
Once again, the last two rows are factorized after a few calculations that we skip, and we get
\begin{align*}
A_r =&\dfrac{2\, c_r^2}{\rho_r} \, [\rho] \, \underline{\Upsilon} \, (\ueta_0^2 +u_r^2 \, |\uceta|^2) \, \gamma_2 \, 
\Big\{ u_\ell \, \ua_\ell \, (u_r \, \ua_r +i\, c_r^2 \, \ueta_0) \, \ubeta_2^- 
+u_\ell \, c_r^2 \, |\uceta|^2 \, (u_r \, \ua_\ell +i\, c_\ell^2 \, \ueta_0) \Big\} \\
&-\dfrac{1}{\rho_r} \, [\rho] \, [u] \, \underline{\Upsilon} \, \gamma_2 \, (i\, \ueta_0 +u_r \, \ubeta_2^-)^2 \, 
u_\ell \, u_r \, \dfrac{\ua_\ell}{\ua_r} \, (u_r \, \ua_r +i\, c_r^2 \, \ueta_0) \, (\ueta_0^2 +u_r^2 \, |\uceta|^2) \, .
\end{align*}
We have thus shown that the `left' contribution $A_\ell$ of $q_2+q_3+q_4$ reads
\begin{equation}
\label{finalAr}
A_r =\dfrac{2\, c_r^2}{\rho_r} \, [\rho] \, \underline{\Upsilon} \, (\ueta_0^2 +u_r^2 \, |\uceta|^2) \, \gamma_2 \, 
\Big\{ u_\ell \, \ua_\ell \, (u_r \, \ua_r +i\, c_r^2 \, \ueta_0) \, \ubeta_2^- 
+u_\ell \, c_r^2 \, |\uceta|^2 \, (u_r \, \ua_\ell +i\, c_\ell^2 \, \ueta_0) \Big\} +\dfrac{c_r^2}{\rho_r} \, Q_r \, .
\end{equation}
\bigskip

The final step of the proof consists in simplifying the remaining terms in $A_\ell$ and $A_r$. More specifically, the 
first factor in the expression \eqref{finalAl} of $A_\ell$ can be simplied as follows:
\begin{align*}
u_\ell \, \Big\{ u_r \, \ua_r \, &(u_\ell \, \ua_\ell +i\, c_\ell^2 \, \ueta_0) \, \ubeta_1^- 
-u_r \, c_\ell^2 \, |\uceta|^2 \, (u_\ell \, \ua_r +i\, c_r^2 \, \ueta_0) \Big\} \\
=\, & u_r \, \ua_r \, (u_\ell \, \ua_\ell +i\, c_\ell^2 \, \ueta_0) \, (u_\ell \, \ubeta_1^- -i\, \ueta_0) 
-c_\ell^2 \, \ueta_0^2 \, (u_r \, \ua_r +i\, c_r^2 \, \ueta_0) 
-u_\ell \, u_r \, c_\ell^2 \, |\uceta|^2 \, (u_\ell \, \ua_r +i\, c_r^2 \, \ueta_0) \\
=\, & u_r \, \ua_r \, c_\ell^2 \, (\ueta_0^2 +u_\ell^2 \, |\uceta|^2) -c_\ell^2 \, \ueta_0^2 \, (u_r \, \ua_r +i\, c_r^2 \, \ueta_0) 
-u_\ell \, u_r \, c_\ell^2 \, |\uceta|^2 \, (u_\ell \, \ua_r +i\, c_r^2 \, \ueta_0) \\
=\, & -i \, c_\ell^2 \, c_r^2 \, \ueta_0 \, (\ueta_0^2 +u_\ell \, u_r \, |\uceta|^2) \, .
\end{align*}
Similarly, the first factor in the expression \eqref{finalAr} of $A_r$ can be simplied by using
\begin{equation*}
u_r \, \Big\{ u_\ell \, \ua_\ell \, (u_r \, \ua_r +i\, c_r^2 \, \ueta_0) \, \ubeta_2^- 
+u_\ell \, c_r^2 \, |\uceta|^2 \, (u_r \, \ua_\ell +i\, c_\ell^2 \, \ueta_0) \Big\} 
=i \, c_\ell^2 \, c_r^2 \, \ueta_0 \, (\ueta_0^2 +u_\ell \, u_r \, |\uceta|^2) \, .
\end{equation*}
Using these two last simplifications, we can add \eqref{finalAl} and \eqref{finalAr} and obtain
\begin{equation*}
(q_2+q_3+q_4)(k,k') =\dfrac{c_\ell^2}{\rho_\ell} \, Q_\ell +\dfrac{c_r^2}{\rho_r} \, Q_r +Q_\sharp \, ,
\end{equation*}
with $Q_\sharp$ defined in \eqref{defQlr}. This completes the proof of Proposition \ref{prop2}.
\end{proof}

\noindent We now compute the kernel $q_2+q_3+q_4$ in the domain $\{ k>0,k'>0,k+k'>0 \}$.

\begin{proposition}
\label{prop3}
Let $Q_\ell$ and $Q_r$ be defined in \eqref{defQlr} and let us define
\begin{equation*}
Q_\flat :=-2\, [\rho] \, [u] \, \underline{\Upsilon} \, (\ueta_0^2 +u_r^2 \, |\uceta|^2) \, u_\ell \, u_r \, |\uceta|^2 \, 
\left\{ \dfrac{c_\ell^4 \, \ua_r}{\rho_\ell \, \ua_\ell} \, \overline{\gamma_1} \, (i\, \ueta_0 -u_\ell \, \ubeta_1^+) 
+\dfrac{c_r^4 \, \ua_\ell}{\rho_r \, \ua_r} \, \overline{\gamma_2} \, (i\, \ueta_0 +u_r \, \ubeta_2^+) \right\} \, .
\end{equation*}
Then the kernels $q_2,q_3,q_4$ defined in \eqref{defq2}, \eqref{defq3} and \eqref{defq4} satisfy
\begin{equation*}
(q_2+q_3+q_4)(k,k') =\left\{ \left( \dfrac{p''(\rho_\ell)}{2} -\dfrac{c_\ell^2}{\rho_\ell} \right) \, \overline{Q_\ell} 
+\left( \dfrac{p''(\rho_r)}{2} -\dfrac{c_r^2}{\rho_r} \right) \, \overline{Q_r} +Q_\flat \right\} \, \left( 1+\dfrac{k'}{k} \right) \, ,
\end{equation*}
for all $(k,k')$ such that $k>0$, $k'<0$ and $k+k'>0$.
\end{proposition}

\begin{proof}
We split again the proof in several steps, as was done in the proof of Proposition \ref{prop2}.
\bigskip

$\bullet$ \underline{Step 1: computation of the $p''(\rho_\ell)$ factor.} We collect again the contributions 
that involve $p''(\rho_\ell)$. The contribution of the kernel $q_2$ equals
\begin{multline}
\label{prop3-1}
\underline{\Upsilon} \, (D_{d+1} +u_\ell \, D_{d+2}) \, |\gamma_1|^2 \, (-i\, \ueta_0 +u_\ell \, \ubeta_1^-) \, 
(i\, \ueta_0 -u_\ell \, \ubeta_1^+) \\
=-[\rho] \, [u] \, \underline{\Upsilon} \, \overline{\gamma_1} \, |i\, \ueta_0 -u_\ell \, \ubeta_1^-|^2 \, \Big\{ 
(\ueta_0^2 +u_r^2 \, |\uceta|^2) \, i \, c_r^2 \, \ueta_0 +(u_r \, \ua_r -i \, c_r^2 \, \ueta_0) \, u_\ell \, u_r \, |\uceta|^2 
\Big\} \, ,
\end{multline}
where the product $(D_{d+1} +u_\ell \, D_{d+2}) \, \gamma_1$ has already been computed when deriving 
\eqref{prop2-1}.

The contribution of the kernel $q_3$ equals
\begin{align*}
& i\, (k+k') \, \int_0^{+\infty} \dfrac{\omega_1}{\gamma_1} \, (-i \, \uceta^T) \, {\rm e}^{-(k+k') \, \ubeta_1^+ \, z} \, 
|\gamma_1|^2 \, |i\, \ueta_0 -u_\ell \, \ubeta_1^-|^2 \, \uceta \, {\rm e}^{k \, \ubeta_1^- \, z} \, 
{\rm e}^{k' \, \ubeta_1^+ \, z} \, \dif z \\
+\, & i\, (k+k') \, \int_0^{+\infty} \dfrac{\omega_3}{\gamma_1} \, (\ueta_0 \, \uceta^T) \, {\rm e}^{-(k+k') \, \ubeta_3^+ \, z} \, 
|\gamma_1|^2 \, |i\, \ueta_0 -u_\ell \, \ubeta_1^-|^2 \, \uceta \, {\rm e}^{k \, \ubeta_1^- \, z} \, 
{\rm e}^{k' \, \ubeta_1^+ \, z} \, \dif z \\
=\, & \overline{\gamma_1} \, |i\, \ueta_0 -u_\ell \, \ubeta_1^-|^2 \, \left\{ \left( 1+\dfrac{k'}{k} \right) \, 
\dfrac{\omega_1 \, |\uceta|^2}{\ubeta_1^+ -\ubeta_1^-} 
+\dfrac{i\, \ueta_0 \, (k+k')\, u_\ell \, |\uceta|^2 \, \omega_3}{k\, (i\, \ueta_0 -u_\ell \, \ubeta_1^-) 
+k' \, (i\, \ueta_0 -u_\ell \, \ubeta_1^+)} \right\} \, .
\end{align*}
Similarly, the contribution of the kernel $q_4$ reads
\begin{align*}
& -\int_0^{+\infty} \dfrac{\omega_1}{\gamma_1} \, \ubeta_1^+ \, {\rm e}^{-(k+k') \, \ubeta_1^+ \, z} \, |\gamma_1|^2 \, 
|i\, \ueta_0 -u_\ell \, \ubeta_1^-|^2 \, (k\, \ubeta_1^- +k' \, \ubeta_1^+) \, {\rm e}^{k \, \ubeta_1^- \, z} \, 
{\rm e}^{k' \, \ubeta_1^+ \, z} \, \dif z \\
& -\int_0^{+\infty} \dfrac{\omega_3}{\gamma_1} \, u_\ell \, |\uceta|^2 \, {\rm e}^{-(k+k') \, \ubeta_3^+ \, z} \, 
|\gamma_1|^2 \, |i\, \ueta_0 -u_\ell \, \ubeta_1^-|^2 \, (k\, \ubeta_1^- +k' \, \ubeta_1^+) \, {\rm e}^{k \, \ubeta_1^- \, z} \, 
{\rm e}^{k' \, \ubeta_1^+ \, z} \, \dif z \\
=\, & -\overline{\gamma_1} \, |i\, \ueta_0 -u_\ell \, \ubeta_1^-|^2 \, \left\{ 
\dfrac{\omega_1}{\ubeta_1^+ -\ubeta_1^-} \, \left( \ubeta_1^+ \, \ubeta_1^- +\dfrac{k'}{k} \, (\ubeta_1^+)^2 \right) 
+\dfrac{(k\, u_\ell \, \ubeta_1^- +k' \, u_\ell \, \ubeta_1^+) \, u_\ell \, |\uceta|^2 \, \omega_3}
{k\, (i\, \ueta_0 -u_\ell \, \ubeta_1^-) +k' \, (i\, \ueta_0 -u_\ell \, \ubeta_1^+)} \right\} \, .
\end{align*}
Adding the contributions of $q_3$ and $q_4$ gives the term
\begin{equation}
\label{prop3-2}
\overline{\gamma_1} \, |i\, \ueta_0 -u_\ell \, \ubeta_1^-|^2 \, \left\{ -\dfrac{1}{2\, \ua_\ell} \, \left( \dfrac{\ua_\ell^2}{c_\ell^2} 
+c_\ell^2 \, |\uceta|^2 \right) \, \omega_1 \, +u_\ell \, |\uceta|^2 \, \omega_3 +\dfrac{k'}{k} \, \omega_1 \, 
\dfrac{|\uceta|^2 -(\ubeta_1^+)^2}{\ubeta_1^+ -\ubeta_1^-} \right\} \, .
\end{equation}

We now add the contributions in \eqref{prop3-1} and \eqref{prop3-2} in order to obtain the $p''(\rho_\ell)$ term in 
$q_2+q_3+q_4$. There is first a constant term that is independent of $(k,k')$, and this term is entirely similar to the 
one derived when adding \eqref{prop2-1} and \eqref{prop2-2} (see Step 1 in the proof of Proposition \ref{prop2}). 
Namely, the constant term in the sum of \eqref{prop3-1} and \eqref{prop3-2} equals
\begin{multline*}
-[\rho] \, [u] \, \underline{\Upsilon} \, \overline{\gamma_1} \, |i\, \ueta_0 -u_\ell \, \ubeta_1^-|^2 \, 
\dfrac{\ueta_0^2 +u_r^2 \, |\uceta|^2}{2\, c_\ell^2 \, (\ueta_0^2 +u_\ell^2 \, |\uceta|^2)} \, 
(u_\ell \, \ua_\ell +i\, c_\ell^2 \, \ueta_0) \, \left( c_r^2 \, \ueta_0^2 +\dfrac{u_r \, \ua_r}{u_\ell \, \ua_\ell} \, 
u_\ell^2 \, c_\ell^2 \, |\uceta|^2 \right) \\
=\dfrac{1}{2} \, [\rho] \, [u] \, \underline{\Upsilon} \, \overline{\gamma_1} \, (-i\, \ueta_0 +u_\ell \, \ubeta_1^+) \, 
(\ueta_0^2 +u_r^2 \, |\uceta|^2) \, \left( c_r^2 \, \ueta_0^2 +\dfrac{u_r \, \ua_r}{u_\ell \, \ua_\ell} \, 
u_\ell^2 \, c_\ell^2 \, |\uceta|^2 \right) =\dfrac{1}{2} \, \overline{Q_\ell} \, .
\end{multline*}
The last contribution in the $p''(\rho_\ell)$ term is the one that depends on $(k,k')$ in \eqref{prop3-2}, that is
\begin{align*}
\dfrac{k'}{k} \, \overline{\gamma_1} \, &|i\, \ueta_0 -u_\ell \, \ubeta_1^-|^2 \, \omega_1 \, 
\dfrac{|\uceta|^2 -(\ubeta_1^+)^2}{\ubeta_1^+ -\ubeta_1^-} \\
&=\dfrac{k'}{k} \, [\rho] \, [u] \, \underline{\Upsilon} \, \overline{\gamma_1} \, |i\, \ueta_0 -u_\ell \, \ubeta_1^-|^2 \, 
\dfrac{\ueta_0^2 +u_r^2 \, |\uceta|^2}{\ueta_0^2 +u_\ell^2 \, |\uceta|^2} \, 
\dfrac{(-u_\ell \, u_r \, \ua_r)}{c_\ell^2} \, (u_\ell \, \ua_\ell -i\, c_\ell^2 \, \ueta_0) \, 
\dfrac{(i\, \ueta_0 -u_\ell \, \ubeta_1^+)^2}{2\, c_\ell^2 \, \ua_\ell} \, (c_\ell^2 -u_\ell^2) \\
&=\dfrac{k'}{2\, k} \, [\rho] \, [u] \, \underline{\Upsilon} \, \overline{\gamma_1} \, (-i\, \ueta_0 +u_\ell \, \ubeta_1^+) \, 
(\ueta_0^2 +u_r^2 \, |\uceta|^2) \, u_\ell \, u_r \, \dfrac{\ua_r}{\ua_\ell} \, 
\dfrac{|i\, \ueta_0 -u_\ell \, \ubeta_1^-|^2}{c_\ell^2} \, (c_\ell^2 -u_\ell^2) =\dfrac{k'}{2\, k} \, \overline{Q_\ell} \, .
\end{align*}
We have thus shown that the $p''(\rho_\ell)$ term in $(q_2+q_3+q_4)(k,k')$ is as claimed in Proposition \ref{prop3}.
\bigskip

$\bullet$ \underline{Step 2: computation of the $p''(\rho_r)$ factor.} The contribution of the kernel $q_2$ equals
\begin{equation}
\label{prop3-3}
-\underline{\Upsilon} \, (D_{d+1} +u_r \, D_{d+2}) \, |\gamma_2|^2 \, |i\, \ueta_0 +u_r \, \ubeta_2^-|^2 
=[\rho] \, [u] \, \underline{\Upsilon} \, \overline{\gamma_2} \, |i\, \ueta_0 +u_r \, \ubeta_2^-|^2 \, 
\dfrac{u_\ell \, \ua_\ell}{c_r^2} \, \Big\{ i \, \ueta_0 \, u_r \, \ua_r -u_r^2 \, c_r^2 \, |\uceta|^2 \Big\} \, .
\end{equation}
The contribution of the kernel $q_3$ equals
\begin{multline*}
i\, (k+k') \, \int_0^{+\infty} \dfrac{\omega_2}{\gamma_2} \, (i \, \uceta^T) \, {\rm e}^{-(k+k') \, \ubeta_2^+ \, z} \, 
|\gamma_2|^2 \, |i\, \ueta_0 +u_r \, \ubeta_2^-|^2 \, \uceta \, {\rm e}^{k \, \ubeta_2^- \, z} \, 
{\rm e}^{k' \, \ubeta_2^+ \, z} \, \dif z \\
=-\overline{\gamma_2} \, |i\, \ueta_0 +u_r \, \ubeta_2^-|^2 \, \omega_2 \, \dfrac{|\uceta|^2}{\ubeta_2^+ -\ubeta_2^-} 
\, \left( 1+\dfrac{k'}{k} \right) \, ,
\end{multline*}
and the contribution of the kernel $q_4$ reads
\begin{multline*}
\int_0^{+\infty} \dfrac{\omega_2}{\gamma_2} \, \ubeta_2^+ \, {\rm e}^{-(k+k') \, \ubeta_2^+ \, z} \, |\gamma_2|^2 \, 
|i\, \ueta_0 +u_r \, \ubeta_2^-|^2 \, (k\, \ubeta_2^-+k' \, \ubeta_2^+) \, {\rm e}^{k \, \ubeta_2^- \, z} \, 
{\rm e}^{k' \, \ubeta_2^+ \, z} \, \dif z \\
=\overline{\gamma_2} \, |i\, \ueta_0 +u_r \, \ubeta_2^-|^2 \, \dfrac{\omega_2}{\ubeta_2^+ -\ubeta_2^-}  \, 
\left( \ubeta_2^+ \, \ubeta_2^- +\dfrac{k'}{k} \, (\ubeta_2^+)^2 \right) \, .
\end{multline*}
Adding the contributions of $q_3$ and $q_4$ gives the term
\begin{equation}
\label{prop3-4}
\overline{\gamma_2} \, |i\, \ueta_0 +u_r \, \ubeta_2^-|^2 \, \left\{ 
-\dfrac{\omega_2}{2\, \ua_r} \, \left( \dfrac{\ua_r^2}{c_r^2} +c_r^2 \, |\uceta|^2 \right) +\dfrac{k'}{k} \, \omega_2 \, 
\dfrac{(\ubeta_2^+)^2 -|\uceta|^2}{\ubeta_2^+ -\ubeta_2^-} \right\} \, .
\end{equation}
When we add the latter term with the expression in \eqref{prop3-3}, we obtain the constant term (independent of 
$(k,k')$):
\begin{equation*}
-[\rho] \, [u] \, \underline{\Upsilon} \, \overline{\gamma_2} \, |i\, \ueta_0 +u_r \, \ubeta_2^-|^2 \, 
(\ueta_0^2 +u_r^2 \, |\uceta|^2) \, \dfrac{i\, u_r \, \ueta_0}{2\, \ua_r} \, (u_\ell \, \ua_\ell +i\, c_\ell^2 \, \ueta_0) 
=\dfrac{1}{2} \, \overline{Q_r} \, .
\end{equation*}
The only term that depends on $(k,k')$ arises in \eqref{prop3-4} and equals
\begin{equation*}
\dfrac{k'}{k} \, \overline{\gamma_2} \, |i\, \ueta_0 +u_r \, \ubeta_2^-|^2 \, \omega_2 \, 
\dfrac{(i\, \ueta_0 +u_r \, \ubeta_2^+)^2}{2\, c_r^2 \, \ua_r} \, (c_r^2 -u_r^2) 
=\dfrac{k'}{2\, k} \, \overline{\gamma_2} \, (i\, \ueta_0 +u_r \, \ubeta_2^+)^2 \, 
\dfrac{\ueta_0^2 +u_r^2 \, |\uceta|^2}{\ua_r} \, \omega_2 =\dfrac{k'}{2\, k} \, \overline{Q_r} \, .
\end{equation*}
The sum of all the terms that involve $p''(\rho_r)$ factorizes as claimed in Proposition \ref{prop3}.
\bigskip

$\bullet$ \underline{Step 3: computation of the remaining terms.} In order to prove Proposition \ref{prop3}, we can 
assume from now on, and without loss of generality that $p''(\rho_\ell) =p''(\rho_r) =0$ in \eqref{hessienne1} and 
\eqref{hessienne2}. With this simplification, we compute
\begin{align}
\sum_{k=1}^{d-1} \uceta_k \, \dif^2 f^k (v_\ell) \cdot (\underline{r}_1^-,\underline{r}_1^+) 
&=-i\, \dfrac{c_\ell^4 \, |\uceta|^2}{\rho_\ell} \, \begin{pmatrix}
0 \\
2\, i \, \uceta \\
-(\ubeta_1^+ +\ubeta_1^-) \end{pmatrix} \, ,\label{prop3-5} \\
\dif^2 \tilde{f}^d (v_\ell) \cdot (\underline{r}_1^-,\underline{r}_1^+) &=\dfrac{c_\ell^4}{\rho_\ell} \, 
\begin{pmatrix}
0 \\
i \, (\ubeta_1^+ +\ubeta_1^-) \, \uceta \\
-2\, \ubeta_1^+ \, \ubeta_1^- \\
-2\, u_\ell \, \ubeta_1^+ \, \ubeta_1^- \end{pmatrix} \, ,\label{prop3-6} \\
\sum_{k=1}^{d-1} \uceta_k \, \dif^2 f^k (v_r) \cdot (\underline{r}_2^-,\underline{r}_2^+) 
&=-i\, \dfrac{c_r^4 \, |\uceta|^2}{\rho_r} \, \begin{pmatrix}
0 \\
2\, i \, \uceta \\
\ubeta_2^+ +\ubeta_2^- \end{pmatrix} \, ,\label{prop3-7} \\
\dif^2 \tilde{f}^d (v_r) \cdot (\underline{r}_2^-,\underline{r}_2^+) &=-\dfrac{c_r^4}{\rho_r} \, 
\begin{pmatrix}
0 \\
i \, (\ubeta_2^+ +\ubeta_2^-) \, \uceta \\
2\, \ubeta_2^+ \, \ubeta_2^- \\
2\, u_r \, \ubeta_2^+ \, \ubeta_2^- \end{pmatrix} \, ,\label{prop3-8}
\end{align}
As was done in Step 3 of the proof of Proposition \ref{prop2}, we are going to split the computations between 
the `left' and `right' contributions. Let us first concentrate on all `left' terms. The `left' contribution of $q_2$ is 
computed from \eqref{prop3-6} and equals
\begin{multline*}
\sigma^* \, \dif^2 \tilde{f}^d (v_\ell) \cdot (\underline{r}_1^-,\underline{r}_1^+) \, |\gamma_1|^2 
=\sigma^* \, \dfrac{c_\ell^4}{\rho_\ell} \, \begin{pmatrix}
0 \\
i \, (\ubeta_1^+ +\ubeta_1^-) \, \uceta \\
-2\, \ubeta_1^+ \, \ubeta_1^- \\
-2\, u_\ell \, \ubeta_1^+ \, \ubeta_1^- \end{pmatrix} \, |\gamma_1|^2 \\
=\dfrac{2\, c_\ell^4}{\rho_\ell \, (c_\ell^2 -u_\ell^2)} \, \underline{\Upsilon}\, \overline{\gamma_1} \, \Big\{ 
u_\ell \, \ueta_0 \, |\uceta|^2 \, \gamma_1 \, \check{D} -(\ueta_0^2 -c_\ell^2 \, |\uceta|^2) \, \gamma_1 
\, (D_{d+1} +u_\ell \, D_{d+2}) \Big\} \, .
\end{multline*}
Using Lemma \ref{lem4}, we find that the `left' contribution of $q_2$ is given by
\begin{equation}
\label{prop3-9}
\dfrac{2\, c_\ell^4}{\rho_\ell \, (c_\ell^2 -u_\ell^2)} \, [\rho] \, [u] \, \underline{\Upsilon}\, \overline{\gamma_1} \, \Big\{ 
i\, c_r^2 \, \ueta_0 \, (\ueta_0^2 -c_\ell^2 \, |\uceta|^2) \, (\ueta_0^2 +u_r^2 \, |\uceta|^2) 
-u_\ell \, u_r \, c_\ell^2 \, |\uceta|^4 \, (u_r \, \ua_r -i\, c_r^2 \, \ueta_0) \Big\} \, .
\end{equation}
Similarly, we use \eqref{prop3-8} to compute the `right' contribution of $q_2$, which yields
\begin{multline}
\label{prop3-10}
-\sigma^* \, \dif^2 \tilde{f}^d (v_\ell) \cdot (\underline{r}_2^-,\underline{r}_2^+) \, |\gamma_2|^2 
=\dfrac{2\, c_r^4}{\rho_r \, (c_r^2 -u_r^2)} \, [\rho] \, [u] \, \underline{\Upsilon} \, \overline{\gamma_2} \\
\times \Big\{ (i\, c_\ell^2 \, \ueta_0^3 +u_\ell \, \ua_\ell \, u_r^2 \, |\uceta|^2) \, (\ueta_0^2 -c_r^2 \, |\uceta|^2) 
-\ueta_0^2 \, u_r^2 \, |\uceta|^2 \, (u_\ell \, \ua_\ell -i\, c_\ell^2 \, \ueta_0) \Big\} \, .
\end{multline}

Using \eqref{prop3-5}, the `left' contribution of $q_3$ equals
\begin{align}
& i\, (k+k') \, \int_0^{+\infty} \dfrac{\omega_1}{\gamma_1} \, \widetilde{\ell}_1 \, {\rm e}^{-(k+k') \, \ubeta_1^+ \, z} \, 
\dfrac{-i\, c_\ell^4 \, |\uceta|^2}{\rho_\ell} \, \begin{pmatrix}
0 \\
2\, i \, \uceta \\
-(\ubeta_1^+ +\ubeta_1^-) \end{pmatrix} \, |\gamma_1|^2 \, {\rm e}^{k \, \ubeta_1^- \, z} \, 
{\rm e}^{k' \, \ubeta_1^+ \, z} \, \dif z \notag \\
+\, & i\, (k+k') \, \int_0^{+\infty} \dfrac{\omega_3}{\gamma_1} \, \widetilde{\ell}_3 \, {\rm e}^{-(k+k') \, \ubeta_3^+ \, z} \, 
\dfrac{-i\, c_\ell^4 \, |\uceta|^2}{\rho_\ell} \, \begin{pmatrix}
0 \\
2\, i \, \uceta \\
-(\ubeta_1^+ +\ubeta_1^-) \end{pmatrix} \, |\gamma_1|^2 \, {\rm e}^{k \, \ubeta_1^- \, z} \, 
{\rm e}^{k' \, \ubeta_1^+ \, z} \, \dif z \notag \\
=\, & \dfrac{c_\ell^4}{\rho_\ell} \, \overline{\gamma_1} \, \omega_1 \, \widetilde{\ell}_1 \, \begin{pmatrix}
0 \\
2\, i \, \uceta \\
-(\ubeta_1^+ +\ubeta_1^-) \end{pmatrix} \, \dfrac{|\uceta|^2}{\ubeta_1^+ -\ubeta_1^-} \, \left( 1+\dfrac{k'}{k} \right) \notag \\
&+\underbrace{\dfrac{c_\ell^4}{\rho_\ell} \, \overline{\gamma_1} \, \omega_3 \, \widetilde{\ell}_1 \, \begin{pmatrix}
0 \\
2\, i \, \uceta \\
-(\ubeta_1^+ +\ubeta_1^-) \end{pmatrix} \, 
\dfrac{(k+k') \, |\uceta|^2}{k\, (\ubeta_3^+ -\ubeta_1^-) +k'\, (\ubeta_3^+ -\ubeta_1^+)}}_{\diamond} \, ,\label{prop3-11}
\end{align}
and, similarly (using now \eqref{prop3-7} rather than \eqref{prop3-5}), the `right' contribution of $q_3$ equals
\begin{multline}
\label{prop3-12}
i\, (k+k') \, \int_0^{+\infty} \dfrac{\omega_2}{\gamma_2} \, \widetilde{\ell}_2 \, {\rm e}^{-(k+k') \, \ubeta_2^+ \, z} \, 
\dfrac{-i\, c_r^4 \, |\uceta|^2}{\rho_r} \, \begin{pmatrix}
0 \\
2\, i \, \uceta \\
\ubeta_2^+ +\ubeta_2^- \end{pmatrix} \, |\gamma_2|^2 \, {\rm e}^{k \, \ubeta_2^- \, z} \, 
{\rm e}^{k' \, \ubeta_2^+ \, z} \, \dif z \\
=\dfrac{c_r^4}{\rho_r} \, \overline{\gamma_2} \, \omega_2 \, \widetilde{\ell}_2 \, \begin{pmatrix}
0 \\
2\, i \, \uceta \\
\ubeta_2^+ +\ubeta_2^- \end{pmatrix} \, \dfrac{|\uceta|^2}{\ubeta_2^+ -\ubeta_2^-} \, \left( 1+\dfrac{k'}{k} \right) \, .
\end{multline}

The `left' contribution of $q_4$ is computed by retaining the three first rows in \eqref{prop3-6}:
\begin{align}
& -\int_0^{+\infty} \dfrac{\omega_1}{\gamma_1} \, \widetilde{\ell}_1 \, {\rm e}^{-(k+k') \, \ubeta_1^+ \, z} \, 
\dfrac{c_\ell^4}{\rho_\ell} \, \begin{pmatrix}
0 \\
i \, (\ubeta_1^+ +\ubeta_1^-) \, \uceta \\
-2\, \ubeta_1^+ \, \ubeta_1^- \end{pmatrix} \, |\gamma_1|^2 \, (k\, \ubeta_1^- +k' \, \ubeta_1^+)\, 
{\rm e}^{k \, \ubeta_1^- \, z} \, {\rm e}^{k' \, \ubeta_1^+ \, z} \, \dif z \notag \\
& -\int_0^{+\infty} \dfrac{\omega_3}{\gamma_1} \, \widetilde{\ell}_1 \, {\rm e}^{-(k+k') \, \ubeta_3^+ \, z} \, 
\dfrac{c_\ell^4}{\rho_\ell} \, \begin{pmatrix}
0 \\
i \, (\ubeta_1^+ +\ubeta_1^-) \, \uceta \\
-2\, \ubeta_1^+ \, \ubeta_1^- \end{pmatrix} \, |\gamma_1|^2 \, (k\, \ubeta_1^- +k' \, \ubeta_1^+)\, 
{\rm e}^{k \, \ubeta_1^- \, z} \, {\rm e}^{k' \, \ubeta_1^+ \, z} \, \dif z \notag \\
=\, & -\dfrac{c_\ell^4}{\rho_\ell} \, \overline{\gamma_1} \, \omega_1 \, \widetilde{\ell}_1 \, \begin{pmatrix}
0 \\
i \, (\ubeta_1^+ +\ubeta_1^-) \, \uceta \\
-2\, \ubeta_1^+ \, \ubeta_1^- \end{pmatrix} \, \dfrac{1}{\ubeta_1^+ -\ubeta_1^-} \, 
\left( \ubeta_1^- +\dfrac{k'}{k} \, \ubeta_1^+ \right) \notag \\
& -\underbrace{\dfrac{c_\ell^4}{\rho_\ell} \, \overline{\gamma_1} \, \omega_3 \, \widetilde{\ell}_1 \, \begin{pmatrix}
0 \\
i \, (\ubeta_1^+ +\ubeta_1^-) \, \uceta \\
-2\, \ubeta_1^+ \, \ubeta_1^- \end{pmatrix} \, 
\dfrac{k\, \ubeta_1^- +k' \, \ubeta_1^+}{k\, (\ubeta_3^+ -\ubeta_1^-) +k'\, (\ubeta_3^+ -\ubeta_1^+)}}_{\diamond} 
\, ,\label{prop3-13}
\end{align}
and the `right' contribution of $q_4$ equals
\begin{multline}
\label{prop3-14}
-\dfrac{c_r^4}{\rho_r} \, \overline{\gamma_2} \, \omega_2 \, \widetilde{\ell}_2 \, \begin{pmatrix}
0 \\
i \, (\ubeta_2^+ +\ubeta_2^-) \, \uceta \\
2\, \ubeta_2^+ \, \ubeta_2^- \end{pmatrix} \, \dfrac{1}{\ubeta_2^+ -\ubeta_2^-} \, 
\left( \ubeta_2^- +\dfrac{k'}{k} \, \ubeta_2^+ \right) \\
=\dfrac{c_r^4}{\rho_r} \, \overline{\gamma_2} \, \omega_2 \, \widetilde{\ell}_2 \, \begin{pmatrix}
0 \\
i (\ubeta_2^+ +\ubeta_2^-) \uceta \\
2\, \ubeta_2^+ \, \ubeta_2^- \end{pmatrix} 
-\dfrac{c_r^4}{\rho_r} \, \overline{\gamma_2} \, \omega_2 \, \widetilde{\ell}_2 \, \begin{pmatrix}
0 \\
i (\ubeta_2^+ +\ubeta_2^-) \uceta \\
2\, \ubeta_2^+ \, \ubeta_2^- \end{pmatrix} \, \dfrac{\ubeta_2^+}{\ubeta_2^+ -\ubeta_2^-} \, 
\left( 1+\dfrac{k'}{k} \right) \, .
\end{multline}

The `left' contribution of $q_3+q_4$ is computed as follows. We first observe that the difference between the 
diamond term in \eqref{prop3-11} and the diamond term in \eqref{prop3-13} reads
\begin{equation*}
\dfrac{c_\ell^2}{\rho_\ell} \, \overline{\gamma_1} \, \omega_3 \, u_\ell \, |\uceta|^2 \, \Big\{ 
c_\ell^2 \, |\uceta|^2 -c_\ell^2 \, \ubeta_1^+ \, \ubeta_1^- -(i\, \ueta_0 -u_\ell \, \ubeta_1^+) \, 
(i\, \ueta_0 -u_\ell \, \ubeta_1^-) \Big\} 
=\dfrac{2\, c_\ell^6}{\rho_\ell \, (c_\ell^2 -u_\ell^2)} \, u_\ell \, |\uceta|^4 \, \overline{\gamma_1} \, \omega_3 \, .
\end{equation*}
Consequently the `left' contribution of $q_3+q_4$, which corresponds to the sum of \eqref{prop3-11} and 
\eqref{prop3-13}, equals
\begin{align*}
&\dfrac{c_\ell^4 \, \overline{\gamma_1} \, \omega_1}{\rho_\ell \, (\ubeta_1^+ -\ubeta_1^-)} \, \left\{ 
|\uceta|^2 \, \widetilde{\ell}_1 \, \begin{pmatrix}
0 \\
2\, i \, \uceta \\
-(\ubeta_1^+ +\ubeta_1^-) \end{pmatrix} -\ubeta_1^+ \, \widetilde{\ell}_1 \, \begin{pmatrix}
0 \\
i \, (\ubeta_1^+ +\ubeta_1^-) \, \uceta \\
-2\, \ubeta_1^+ \, \ubeta_1^- \end{pmatrix} \right\} \, \left( 1+\dfrac{k'}{k} \right) \\
&+\dfrac{c_\ell^4}{\rho_\ell} \, \overline{\gamma_1} \, \omega_1 \, \widetilde{\ell}_1 \, \begin{pmatrix}
0 \\
i \, (\ubeta_1^+ +\ubeta_1^-) \, \uceta \\
-2\, \ubeta_1^+ \, \ubeta_1^- \end{pmatrix} +\dfrac{2\, c_\ell^6}{\rho_\ell \, (c_\ell^2 -u_\ell^2)} 
\, u_\ell \, |\uceta|^4 \, \overline{\gamma_1} \, \omega_3 \\
=& \dfrac{2\, c_\ell^4}{\rho_\ell \, (\ubeta_1^+ -\ubeta_1^-)} \, \overline{\gamma_1} \, \omega_1 \, 
(\ubeta_1^+ \, \ubeta_1^- -|\uceta|^2) \, ((\ubeta_1^+)^2 -|\uceta|^2) \, \left( 1+\dfrac{k'}{k} \right) \\
&-\dfrac{2\, c_\ell^4}{\rho_\ell \, (c_\ell^2 -u_\ell^2)} \, \overline{\gamma_1} \, \omega_1 \, \Big\{ 
(\ueta_0^2 -c_\ell^2 \, |\uceta|^2) \, \ubeta_1^+ +i\, u_\ell \, \ueta_0 \,  |\uceta|^2 \Big\} 
+\dfrac{2\, c_\ell^6}{\rho_\ell \, (c_\ell^2 -u_\ell^2)} \, u_\ell \, |\uceta|^4 \, \overline{\gamma_1} \, \omega_3 \, .
\end{align*}
After a little bit of algebra, the `left' contribution of $q_3+q_4$ is found to be equal to
\begin{align*}
&-\dfrac{c_\ell^2}{\rho_\ell} \, [\rho] \, [u] \, \underline{\Upsilon} \, u_\ell \, u_r \, \dfrac{\ua_r}{\ua_\ell} \, 
(\ueta_0^2 +u_r^2 \, |\uceta|^2) \, \left( \dfrac{\ua_\ell^2}{c_\ell^2} +c_\ell^2 \, |\uceta|^2 \right) \, 
\overline{\gamma_1} \, (i\, \ueta_0 -u_\ell \, \ubeta_1^+) \, \left( 1+\dfrac{k'}{k} \right) \\
&+\dfrac{2\, c_\ell^4}{\rho_\ell \, (c_\ell^2 -u_\ell^2)} \, [\rho] \, [u] \, \underline{\Upsilon} \, \overline{\gamma_1} \, 
\Big\{ -i\, c_r^2 \, \ueta_0 \, (\ueta_0^2 -c_\ell^2 \, |\uceta|^2) \, 
(\ueta_0^2 +u_r^2 \, |\uceta|^2) +u_\ell \, u_r \, c_\ell^2 \, |\uceta|^4 \, (u_r \, \ua_r -i\, c_r^2 \, \ueta_0) \Big\} \, .
\end{align*}
When combined with \eqref{prop3-9}, we have thus shown that the `left' contribution of the kernel $q_2+q_3+q_4$ 
reads
\begin{multline}
\label{prop3-15}
-\dfrac{c_\ell^2}{\rho_\ell} \, [\rho] \, [u] \, \underline{\Upsilon} \, u_\ell \, u_r \, \dfrac{\ua_r}{\ua_\ell} \, 
(\ueta_0^2 +u_r^2 \, |\uceta|^2) \, \left( \dfrac{\ua_\ell^2}{c_\ell^2} +c_\ell^2 \, |\uceta|^2 \right) \, 
\overline{\gamma_1} \, (i\, \ueta_0 -u_\ell \, \ubeta_1^+) \, \left( 1+\dfrac{k'}{k} \right) \\
=-\left( \dfrac{c_\ell^2}{\rho_\ell} \, \overline{Q_\ell} +2\, [\rho] \, [u] \, \underline{\Upsilon} \, u_\ell \, u_r \, |\uceta|^2 
\, (\ueta_0^2 +u_r^2 \, |\uceta|^2) \, \dfrac{c_\ell^4 \, \ua_r}{\rho_\ell \, \ua_\ell} \, \overline{\gamma_1} \, 
(i\, \ueta_0 -u_\ell \, \ubeta_1^+) \right) \, \left( 1+\dfrac{k'}{k} \right) \, .
\end{multline}
\bigskip

Let us now compute the `right' contribution of the kernel $q_3+q_4$. We add the expressions in \eqref{prop3-12} 
and \eqref{prop3-14}, which gives
\begin{equation*}
\dfrac{2\, c_r^4}{\rho_r} \, \overline{\gamma_2} \, \omega_2 \, \left\{ (\ubeta_2^+ \, \ubeta_2^-) \, \ubeta_2^+ 
-\dfrac{\ubeta_2^+ +\ubeta_2^-}{2} \, |\uceta|^2 \right\} 
-\dfrac{2\, c_r^4}{\rho_r} \, \overline{\gamma_2} \, \omega_2 \, ((\ubeta_2^+)^2 -|\uceta|^2) \, 
\dfrac{\ubeta_2^+ \, \ubeta_2^- -|\uceta|^2}{\ubeta_2^+ -\ubeta_2^-} \, \left( 1+\dfrac{k'}{k} \right) \, ,
\end{equation*}
or equivalently
\begin{multline}
\label{prop3-16}
\dfrac{2\, c_r^4}{\rho_r \, (c_r^2 -u_r^2)} \, [\rho] \, [u] \, \underline{\Upsilon} \, \overline{\gamma_2} \, i \, u_r 
\, \ueta_0 \, (u_\ell \, \ua_\ell -i\, c_\ell^2 \, \ueta_0) \, \Big\{ (\ueta_0^2 -c_r^2 \, |\uceta|^2) \, \ubeta_2^+ 
-i \, u_r \, \ueta_0 \, |\uceta|^2 \Big\} \\
-\dfrac{c_r^4}{\rho_r} \, [\rho] \, [u] \, \underline{\Upsilon} \, u_\ell \, u_r \, \dfrac{\ua_\ell}{\ua_r} \, 
(\ueta_0^2 +u_r^2 \, |\uceta|^2) \, \left( \dfrac{\ua_r^2}{c_r^2} +c_r^2 \, |\uceta|^2 \right) \, 
\overline{\gamma_2} \, (i\, \ueta_0 +u_r \, \ubeta_2^+) \, \left( 1+\dfrac{k'}{k} \right) \, .
\end{multline}
The first row in \eqref{prop3-16} is exactly the opposite of \eqref{prop3-10}, that is of the `right' contribution 
of $q_2$. In other words, we have found that the `right' contribution of $q_2+q_3+q_4$ is given by the second 
row in \eqref{prop3-16}, which reads
\begin{equation*}
-\left( \dfrac{c_r^2}{\rho_r} \, \overline{Q_r} +2\, [\rho] \, [u] \, \underline{\Upsilon} \, u_\ell \, u_r \, |\uceta|^2 
\, (\ueta_0^2 +u_r^2 \, |\uceta|^2) \, \dfrac{c_r^4 \, \ua_\ell}{\rho_r \, \ua_r} \, \overline{\gamma_2} \, 
(i\, \ueta_0 +u_r \, \ubeta_2^+) \right) \, \left( 1+\dfrac{k'}{k} \right) \, .
\end{equation*}
The kernel $q_2+q_3+q_4$ is the sum of the latter quantity and the right hand side of \eqref{prop3-15}. Collecting 
the terms, this completes the proof of Proposition \ref{prop3}.
\end{proof}

\begin{corollary}
\label{cor2}
With $Q$ defined in \eqref{defQ}, $Q_\ell,Q_r,Q_\sharp,Q_\flat$ defined in Propositions \ref{prop2} and \ref{prop3}, 
the kernel $q$ satisfies
\begin{equation*}
q(k,k') =\begin{cases}
\left( \dfrac{p''(\rho_\ell)}{2} +\dfrac{c_\ell^2}{\rho_\ell} \right) \, Q_\ell 
+\left( \dfrac{p''(\rho_r)}{2} +\dfrac{c_r^2}{\rho_r} \right) \, Q_r +Q_\sharp &\text{\rm if $k>0$ and $k'>0$,}\\
\left\{ \! \left( \dfrac{p''(\rho_\ell)}{2} -\dfrac{c_\ell^2}{\rho_\ell} \right) \, \overline{Q_\ell} 
+\left( \dfrac{p''(\rho_r)}{2} -\dfrac{c_r^2}{\rho_r} \right) \, \overline{Q_r} +Q_\flat +\overline{Q} \right\} 
\left( 1+\dfrac{k'}{k} \right) &\text{\rm if $k>0, k'<0, k+k'>0$.}
\end{cases}
\end{equation*}
\end{corollary}

\section{Conclusion}

There remains a final simplification in order to achieve the final form of the kernel $q$. Our main result reads 
as follows.

\begin{proposition}
\label{prop4}
With $Q_\ell,Q_r,Q_\sharp$ defined in Propositions \ref{prop2}, let us define
\begin{equation*}
Q_\natural := \left( \dfrac{p''(\rho_\ell)}{2} +\dfrac{c_\ell^2}{\rho_\ell} \right) \, Q_\ell 
+\left( \dfrac{p''(\rho_r)}{2} +\dfrac{c_r^2}{\rho_r} \right) \, Q_r +Q_\sharp \, .
\end{equation*}
Then the kernel $q=4\, \pi \, a_1$ satisfies
\begin{equation*}
q(k,k') =\begin{cases}
Q_\natural &\text{\rm if $k>0$ and $k'>0$,}\\
\overline{Q_\natural} \, \left( 1+\dfrac{k'}{k} \right) &\text{\rm if $k>0$, $k'<0$ and $k+k'>0$.}
\end{cases}
\end{equation*}
In particular, $q$ satisfies Hunter's stability condition $q(1,0^+) =\overline{q(1,0^-)}$, and \eqref{BurgersEuler} 
is locally well-posed in $H^2(\R)$.
\end{proposition}

\begin{proof}
In view of the expression of $q$ given in Corollary \ref{cor2}, we only need to show the relation
\begin{equation*}
\dfrac{c_\ell^2}{\rho_\ell} \, Q_\ell +\dfrac{c_r^2}{\rho_r} \, Q_r +\dfrac{1}{2} \, Q_\sharp 
=\dfrac{1}{2} \, \Big( Q +\overline{Q_\flat} \Big) \, ,
\end{equation*}
with $Q_\ell,Q_r,Q_\sharp$ given in Proposition \ref{prop2}, $Q_\flat$ in Proposition \ref{prop3}, and $Q$ in 
Lemma \ref{lem1}. 
Simplifying by the common factor $[\rho] \, [u] \, \underline{\Upsilon} \, (\ueta_0^2 +u_r^2 \, |\uceta|^2) \, 
u_\ell \, u_r$, we are led to showing that the following relation holds:
\begin{multline}
\label{theo1}
\dfrac{\ua_r}{\rho_\ell \, \ua_\ell} \, c_\ell^2 \, (\ueta_0^2 +u_\ell^2 \, |\uceta|^2) \, \gamma_1 \, 
(i\, \ueta_0 -u_\ell \, \ubeta_1^-) +\dfrac{\ua_\ell}{\rho_r\, \ua_r} \, c_r^2 \, (\ueta_0^2 +u_r^2 \, |\uceta|^2) 
\, \gamma_2 \, (i\, \ueta_0 +u_r \, \ubeta_2^-) \\
-\dfrac{1}{{\bf j} \, [u] \, \ueta_0} \, (\ueta_0^2 +u_\ell \, u_r \, |\uceta|^2) \, i\, \ua_\ell \, \ua_r \, 
(c_r^2 \, \gamma_2 -c_\ell^2 \, \gamma_1) =(\ubeta_1^- +\ubeta_2^-) \, i\, \ua_\ell \, \ua_r \, 
\dfrac{u_\ell \, \ua_r +i\, c_r^2 \, \ueta_0}{u_\ell \, \ua_r -i\, c_r^2 \, \ueta_0} \\
+\dfrac{\ua_r}{\rho_\ell \, \ua_\ell} \, c_\ell^4 \, |\uceta|^2 \, \gamma_1 \, (i\, \ueta_0 -u_\ell \, \ubeta_1^-) 
+\dfrac{\ua_\ell}{\rho_r \, \ua_r} \, c_r^4 \, |\uceta|^2 \, \gamma_2 \, (i\, \ueta_0 +u_r \, \ubeta_2^-) \, ,
\end{multline}
where we use the notation ${\bf j} := \rho_\ell \, u_\ell =\rho_r \, u_r$ to denote the mass flux across the phase 
boundary.

The verification of \eqref{theo1} proceeds as follows. We first combine the first and third row in \eqref{theo1} 
by recalling the definitions of $\ua_\ell,\ua_r$, see \eqref{valeurspropres1} and \eqref{valeurspropres2}. Thus 
verifying \eqref{theo1} amounts to showing
\begin{multline}
\label{theo2}
(\ubeta_1^- +\ubeta_2^-) \, \dfrac{u_\ell \, \ua_r +i\, c_r^2 \, \ueta_0}{u_\ell \, \ua_r -i\, c_r^2 \, \ueta_0} 
-i\, \dfrac{\gamma_1}{\rho_\ell} \, (i\, \ueta_0 -u_\ell \, \ubeta_1^-) 
-i\, \dfrac{\gamma_2}{\rho_r} \, (i\, \ueta_0 +u_\ell \, \ubeta_2^-) \\
+\dfrac{1}{{\bf j} \, [u] \, \ueta_0} \, (\ueta_0^2 +u_\ell \, u_r \, |\uceta|^2) \, 
(c_r^2 \, \gamma_2 -c_\ell^2 \, \gamma_1) =0 \, .
\end{multline}
Let us now define
\begin{align*}
&B_\ell :=\ubeta_1^- \, \dfrac{u_\ell \, \ua_r +i\, c_r^2 \, \ueta_0}{u_\ell \, \ua_r -i\, c_r^2 \, \ueta_0} 
-i\, \dfrac{\gamma_1}{\rho_\ell} \, (i\, \ueta_0 -u_\ell \, \ubeta_1^-) -\dfrac{1}{{\bf j} \, [u] \, \ueta_0} \, 
(\ueta_0^2 +u_\ell \, u_r \, |\uceta|^2) \, c_\ell^2 \, \gamma_1 \, ,\\
&B_r :=\ubeta_2^- \, \dfrac{u_\ell \, \ua_r +i\, c_r^2 \, \ueta_0}{u_\ell \, \ua_r -i\, c_r^2 \, \ueta_0} 
-i\, \dfrac{\gamma_2}{\rho_r} \, (i\, \ueta_0 +u_\ell \, \ubeta_2^-) +\dfrac{1}{{\bf j} \, [u] \, \ueta_0} \, 
(\ueta_0^2 +u_\ell \, u_r \, |\uceta|^2) \, c_r^2 \, \gamma_2 \, ,
\end{align*}
so that \eqref{theo2}, which is the relation we wish to prove, reads $B_\ell +B_r=0$.

Using the relations
\begin{equation*}
\dfrac{u_\ell \, \ua_r +i\, c_r^2 \, \ueta_0}{u_\ell \, \ua_r -i\, c_r^2 \, \ueta_0} 
=-\dfrac{u_r \, \ua_\ell +i\, c_\ell^2 \, \ueta_0}{u_r \, \ua_\ell -i\, c_\ell^2 \, \ueta_0} \, ,\quad 
(u_r \, \ua_\ell -i\, c_\ell^2 \, \ueta_0) \, \gamma_1 =-\rho_\ell \, [u] \, \ueta_0 \, ,
\end{equation*}
we get
\begin{equation}
\label{theo3}
u_\ell \, (u_r \, \ua_\ell -i\, c_\ell^2 \, \ueta_0) \, B_\ell =-i\, \ueta_0 \, (u_\ell \, \ua_\ell +i\, c_\ell^2 \, \ueta_0) \, .
\end{equation}
Similarly, we compute
\begin{equation}
\label{theo4}
u_r \, (u_\ell \, \ua_r -i\, c_r^2 \, \ueta_0) \, B_r =-i\, \ueta_0 \, (u_r \, \ua_r +i\, c_r^2 \, \ueta_0) \, .
\end{equation}
Combining the relations \eqref{theo3} and \eqref{theo4}, we get
\begin{align*}
u_\ell \, u_r \, &(u_r \, \ua_\ell -i\, c_\ell^2 \, \ueta_0) \, (u_\ell \, \ua_r -i\, c_r^2 \, \ueta_0) \, (B_\ell +B_r) \\
&=-i\, \ueta_0 \, u_r \, (u_\ell \, \ua_r -i\, c_r^2 \, \ueta_0) \, (u_\ell \, \ua_\ell +i\, c_\ell^2 \, \ueta_0) 
-i\, \ueta_0 \, u_\ell \, (u_r \, \ua_\ell -i\, c_\ell^2 \, \ueta_0) \, (u_r \, \ua_r +i\, c_r^2 \, \ueta_0) \\
&=-i\, \ueta_0 \, (u_\ell+u_r) \, (u_\ell \, u_r \, \ua_\ell \, \ua_r +c_\ell^2 \, c_r^2 \, \ueta_0^2) =0 \, .
\end{align*}
This means that \eqref{theo2} holds, and consequently the expression of the kernel $q$ is as claimed 
in Proposition \ref{prop4}. The verification of Hunter's stability condition $q(1,0^+) =\overline{q(1,0^-)}$ is then straightforward, and local well-posedness in $H^2(\R)$ follows from the main result in \cite{Benzoni2009}.
\end{proof}

\bibliographystyle{plain}
\bibliography{BC2}
\end{document}